\documentclass[twoside,11pt]{article}

\usepackage{booktabs}  
\usepackage{jmlr2e}
\usepackage{pifont}
\usepackage{graphicx}  
\usepackage{amssymb}   
\usepackage{xcolor}    
\usepackage{graphicx}
\usepackage{amsmath,amssymb,amstext}
\usepackage{mathrsfs}
\usepackage{multirow}
\usepackage{algorithm}
\usepackage{algorithmic}
\usepackage{subfigure}
\usepackage{color}
\usepackage{booktabs}

\newtheorem{assumption}{Assumption}
\newtheorem{lem}{Lemma}
\newtheorem{thm}{Theorem}

\usepackage{lastpage}
\jmlrheading{?}{2025}{1-\pageref{LastPage}}{?; Revised ?}{?}{?}{Junnan Yang, Huiling Zhang, and Zi Xu}

\ShortHeadings{Completely Parameter-Free Algorithms for Minimax Problems}{Completely Parameter-Free Algorithms for Minimax Problems}
\firstpageno{1}

\begin{document}
	\title{Completely Parameter-Free Single-Loop Algorithms for Nonconvex-Concave Minimax Problems\footnotemark[1]}	
	\author{\name Junnan Yang \email yangjunnan@shu.edu.cn \\
		\addr Department of Mathematics\\
		Shanghai University\\
		Shanghai 200444, People's Republic of China
		\AND
		\name Huiling Zhang \email zhanghl@amss.ac.cn \\
		\addr Academy of Mathematics and Systems Science\\ 
		Chinese Academy of Sciences\\
		Beijing 100190, People's Republic of China
		\AND
		\name Zi Xu\footnotemark[2] \email xuzi@shu.edu.cn \\
		\addr Department of Mathematics\\
		Shanghai University\\
		Shanghai 200444, People's Republic of China\\
		and\\
		Newtouch Center for Mathematics of Shanghai University\\
		Shanghai University\\
		Shanghai 200444, People's Republic of China}
	
	\renewcommand{\thefootnote}{\fnsymbol{footnote}}
	\footnotetext[1]{This work is supported by National Natural Science Foundation of China under the grants 12471294 and the Postdoctoral Fellowship Program of CPSF under Grant Number GZB20240802 and 2024M763470.}
	\footnotetext[2]{Corresponding author.}
	\editor{My editor}
	\maketitle
	
	\begin{abstract}
		Due to their importance in various emerging applications, efficient algorithms for solving minimax problems have recently received increasing attention. However, many existing algorithms require prior knowledge of the problem parameters in order to achieve optimal iteration complexity. In this paper, three completely parameter-free single-loop algorithms, namely {PF-AGP-NSC} algorithm, {PF-AGP-NC} algorithm and {PF-AGP-NL} algorithm, are proposed to solve the smooth nonconvex-strongly concave, nonconvex-concave minimax problems and nonconvex-linear minimax problems respectively using line search without requiring any prior knowledge about parameters such as the Lipschtiz constant $L$ or the strongly concave modulus $\mu$. Furthermore, we prove that the total number of gradient calls required to obtain an $\varepsilon$-stationary point for the {PF-AGP-NSC} algorithm, the {PF-AGP-NC} algorithm, and the {PF-AGP-NL} algorithm are upper bounded by \ {$\mathcal{O}\left( L^2\kappa^3\varepsilon^{-2} \right)$}, {$\mathcal{O}\left( \log^2(L)L^4\varepsilon^{-4} \right)$}, and {$\mathcal{O}\left( L^3\varepsilon^{-3} \right)$}, respectively, where $\kappa$ is the condition number.
		To the best of our knowledge, PF-AGP-NC and PF-AGP-NL are the first completely parameter-free algorithms for solving nonconvex-concave and nonconvex-linear minimax problems, respectively. PF-AGP-NSC is a completely parameter-free algorithm for solving nonconvex-strongly concave minimax problems, achieving the best known complexity with respect to $\varepsilon$. Numerical results demonstrate the efficiency of the three proposed algorithms.
	\end{abstract}
	
	
	\section{Introduction}
	We consider the following minimax optimization problem: 
	\begin{equation} \label{p}
		\min_{x \in \mathcal{X}} \max_{y \in \mathcal{Y}} f(x, y), \tag{P}
	\end{equation}
	where $\mathcal{X} \subseteq \mathbb{R}^n$ and $\mathcal{Y} \subseteq \mathbb{R}^m$ are non-empty closed compact convex sets, and $ f:\mathcal{X} \times \mathcal{Y} \rightarrow \mathbb{R} $ is a smooth function. In recent years, the minimax problems have attracted the attention of scholars in the cross-disciplinary fields of optimization and machine learning due to its wide application. Many practical problems can be modeled as \eqref{p}, such as generative adversarial networks (GAN) \citep{Arjovsky2017,Goodfellow2014,Maziar}, reinforcement learning \citep{Dai2018}, power control and transceiver design in signal processing \cite{Lu}, distributed non-convex optimization \citep{Giannakis,Giordano,Mateos}, robust optimization \citep{Ben-Tal2009,Gao2023}, statistical learning \citep{Giordano,ShafieezadehAbadeh2015}, etc. When $f(x, y)$ is (non)convex with respect to $x$ and (non)concave with respect to $y$, we call \eqref{p} the (non)convex-(non)concave minimax problem respectively.
	
	There are many research results on convex-concave minimax problems. The mirror neighbor method was proposed in \citep{Nemirovski2004}, and the dual extrapolation algorithm was proposed in \citep{Nesterov2007}, both of which achieve the optimal iteration complexity $\mathcal{O}(\varepsilon^{-1})$ for solving convex-concave minimax problems. For more related research results, please refer to \citep{Vandenberghe,Chen2014,Chen2017,Dang2014,Monteiro2010,Nedic2009,Yang2020} and their references.
	
	Recently, various algorithms have been proposed to solve the nonconvex-(strongly) concave minimax problems. For the nonconvex-strongly concave minimax problems, some recent studies \citep{Jin,Lin2019,Rafique,Lu} proposed various algorithms that achieve gradient complexity of $\tilde{\mathcal{O}}(\kappa_y^2 \varepsilon^{-2})$ for stationary points of $\Phi(\cdot) = \max_{y \in \mathcal{Y}} f(\cdot, y)$ when $\mathcal{X} = \mathbb{R}^n$ or for stationary points of $f$ where $\kappa_y$ is the condition number of $f(x, \cdot)$. An acceleration algorithm was proposed in \citep{Lin2020}, which improves the gradient complexity to $\tilde{\mathcal{O}}(\sqrt{\kappa_y} \varepsilon^{-2})$. In addition, Zhang et al. \citep{Zhang2021} proposed a general acceleration framework, which improves the iteration complexity to $\mathcal{O}\left(\sqrt{\kappa_y}\varepsilon ^{-2} \right)$.
	For general nonconvex-concave minimax problems, there are two types of algorithms, i.e., multi-loop algorithms and single-loop algorithms. Various multi-loop algorithms have been proposed in \citep{Kong,Lin2020,Nouiehed2019,Ostro,Rafique,Thekumparampil2019}. The best known iteration complexity of multi-loop algorithms for solving nonconvex-concave minimax problems is $\tilde{\mathcal{O}}( \varepsilon ^{-2.5})$, which was achieved by \citep{Lin2020}.  For solving nonconvex-linear minimax problems, Pan et al. \citep{Pan2021} proposed a new alternating gradient projection algorithm with the iteration complexity of $\mathcal{O}( \varepsilon ^{-3})$.
	Various single-loop algorithms have also been proposed  for solving nonconvex-concave minimax problems,  e.g., the gradient descent-ascent (GDA) algorithm \citep{Lin2019}, the hybrid block successive approximation (HiBSA) algorithm \citep{Lu}, the unified single-loop alternating gradient projection (AGP) algorithm \citep{xu2023unified}, and the smoothed GDA algorithm \citep{Zhang2020}. Both the AGP algorithm and the smoothed GDA algorithm achieve the best known iteration complexity, i.e., $\mathcal{O}( \varepsilon ^{-4})$, among single-loop algorithms.

	It should be noted that for most of the above existing results, achieving the optimal complexity of the algorithm requires assuming accurate information about certain parameters of the known problem, such as the Lipschitz constant or the strong concavity modulus $\mu$, etc. Accurately estimating these parameters is usually challenging, and conservative estimates will significantly affect the algorithm performance \citep{lan2023optimal}. Therefore, designing parameter-free algorithms that do not rely on these parameter inputs and have complexity guarantees has attracted widespread attention in recent years.

	For nonconvex-strongly concave minimax problems, there are two main types of parameter-free algorithms: one based on AdaGrad steps and the other based on line search. In the framework of AdaGrad steps, the NeAda algorithm was proposed in \citep{Yang2022}, which obtains an $\varepsilon$-stationary point with an iteration complexity of $\tilde{\mathcal{O}}(\varepsilon^{-2})$ and $\tilde{\mathcal{O}}(\varepsilon^{-4})$ in deterministic and stochastic settings, respectively. The TiAda algorithm was proposed in \citep{Li2022}, which obtains an $\epsilon$-stationary point with an iteration complexity of ${\mathcal{O}}(\varepsilon^{-2})$ and ${\mathcal{O}}(\varepsilon^{-4})$ in deterministic and stochastic settings, respectively. Based on the line search framework, a stochastic gradient descent with backtracking (SGDA-B) algorithm was proposed in \citep{Xu2024}, which has an upper bound of $\mathcal{O}(L\kappa^2\log(\kappa) \varepsilon^{-2})$ on the total number of gradient calls to obtain an $\varepsilon$-stationary point. A proximal alternating gradient descent (AGDA+) method was proposed in \citep{Zhang2024}  which achieves an upper bound of the total number of gradient calls of  $\mathcal{O}(\bar{L}^2\bar{\kappa}^4\log(\mathcal{R})\varepsilon^{-2})$, where $\bar{\kappa} = \bar{L}/\bar{\mu}$, $\bar{L} = \mathcal{R} \tilde{l}$ and $\bar{\mu} = \tilde{\mu}/\mathcal{R}$, $\mathcal{R} = \max\{ \tilde{\mu}/\mu, L/\tilde{l}, 1 \}$, $\tilde{\mu} \geq \mu$ and $\tilde{l} \leq L$ denote the initial estimate of the concavity modulus $\mu$ and the local Lipschitz constant $L$.
	
	For the nonconvex-concave minimax problems, a stochastic GDA method with backtracking (SGDA-B) is proposed in \citep{Xu2024}, which obtains an upper bound of $\tilde{\mathcal{O}}(L^3\varepsilon^{-4})$ on the total number of gradient calls to obtain an $\varepsilon$-stationary point.
	
	For nonconvex-linear minimax problems (a special class of nonconvex-concave minimax problems), before the results of this paper, there were no relevant results on parameter-free optimization algorithms with an iteration complexity of $\mathcal{O}(\varepsilon^{-3})$.
	
	\subsection{Contributions}
	Based on the line search framework, this paper proposes three completely parameter-free alternating gradient projection algorithms, namely PF-AGP-NSC algorithm, PF-AGP-NC algorithm and PF-AGP-NL algorithm, which are used to solve nonconvex-strongly concave, nonconvex-concave and nonconvex-linear minimax problems respectively. These algorithms do not require prior knowledge of any problem-related parameters, including the Lipschtiz constant $L$ and the strongly concave modulus $\mu$. Furthermore, we prove that the total number of gradient calls required to obtain an $\varepsilon$-stationary point for the {PF-AGP-NSC} algorithm, the {PF-AGP-NC} algorithm, and the {PF-AGP-NL} algorithm are upper bounded by \ {$\mathcal{O}\left( L^2\kappa^3\varepsilon^{-2} \right)$}, {$\mathcal{O}\left( \log^2(L)L^4\varepsilon^{-4} \right)$}, and {$\mathcal{O}\left( L^3\varepsilon^{-3} \right)$}, respectively.
	
	Table \ref{tab} provides a detailed comparison of the three proposed parameter-free optimization algorithms with existing parameter-free optimization algorithms for solving nonconvex-strongly concave (NC-SC), nonconvex-concave (NC-C), and nonconvex-linear (NC-L) minimax problems. Column ``\textbf{Completely PF}" indicates whether the algorithm is completely parameter-free (yes/no), and column ``\textbf{Prior Information}" specifies any prior knowledge or parameters required by the algorithm, where $\underline{S} = \min_{x\in\mathcal{X}}\max_{y\in \mathcal{Y}}f(x,y)$ and $ \mathcal{D}_y = \max\{\|y_1 - y_2\| \mid y_1, y_2 \in \mathcal{Y}\} $ and "None" indicates that no prior information is required. In ``Gradient Complexity" column, $\bar{\kappa}\geq \kappa$, $\mathcal{R}$ is a constant which depends on $L$ and $\mu$ and $\mathcal{R}\leq \kappa$.
	
	To the best of our knowledge, the PF-AGP-NL algorithm is the first completely parameter-free single-loop algorithm for solving nonconvex-linear minimax problems, achieving the best known gradient complexity within the single-loop algorithms. 
	The PF-AGP-NSC and the PF-AGP-NC algorithm is the first completely parameter-free single-loop algorithm for solving nonconvex-strongly concave and nonconvex-concave minimax problems, achieving the best known gradient complexity with respect to $\varepsilon$ with the existing single-loop algorithms respectively. 
	
	It should be pointed out that the complexity of the proposed PF-AGP-NSC and PF-AGP-NC algorithms is slightly worse than that of the SGDA-B algorithm in terms of the Lipschitz constant L and the condition number $\kappa$, but the SGDA-B algorithm requires some prior information about the remaining parameters and is not a completely parameter-free algorithm. If some other prior information is also allowed, the complexity of the proposed algorithm can also be improved. To better illustrate this point, if two additional parameters $\mu$ and $\underline{S}$ are known in advance, in Appendix \ref{app1} we propose a restarted variant of the PF-AGP-NSC algorithm, named \text{rPF-AGP-SC}, and prove that the gradient complexity of the algorithm is improved to $\mathcal{O}\left(\log(L) L\kappa^2\varepsilon^{-2}\right)$, which is better than that of the SGDA-B. 
	\begin{table}[]
		\label{tab}
		\resizebox{\textwidth}{!}{ %
			\begin{tabular}{|l|l|l|l|l|}
				\hline \textbf{Problems}
				& \textbf{Algorithms}  & \textbf{Gradient Complexity}                                                    & \textbf{Completely PF} & \textbf{Prior Information}       \\ \hline
				\multirow{5}{*}{\textbf{NC-SC}} & NeAda      & $\mathcal{O}(\log(\kappa)\kappa^{4.5}L^4 \varepsilon^{-2})$      & Yes           & None                    \\ \cline{2-5} 
				& TiAda      & $\mathcal{O}(\kappa^{10}\varepsilon^{-2})$                       & Yes           & None                    \\ \cline{2-5} 
				& SGDA-B     & $\mathcal{O}(L\kappa^2\log(\kappa)\varepsilon^{-2})$          & No            & $\mu$, $\underline{S}$ \\ \cline{2-5} 
				& AGDA+      & $\mathcal{O}(\bar{L}^2\bar{\kappa}^4\log(\mathcal{R})\varepsilon^{-2})$                       & Yes            & None                 \\ \cline{2-5} 
				& \textbf{PF-AGP-NSC} & $\mathcal{O} ( L^2\kappa^3\varepsilon^{-2})$                  & Yes           & None                    \\ \hline
				\multirow{2}{*}{\textbf{NC-C}}  & SGDA-B     & $\mathcal{O}(L^3\log({L}{\varepsilon^{-1}})\varepsilon^{-4})$ & No            & $ \mathcal{D}_y$        \\ \cline{2-5} 
				& \textbf{PF-AGP-NC}  & $\mathcal{O} ( \log^2(L) L^4\varepsilon^{-4})$                & Yes           & None                    \\ \hline
				\textbf{NC-L}                   & \textbf{PF-AGP-NL}  & $\mathcal{O}(L^3 \varepsilon^{-3})$                           & Yes           & None                    \\ \hline
			\end{tabular}
		}
		\caption{Comparison of Gradient Complexity of Existing Algorithms for Nonconvex-Concave Minimax Problems.}
	\end{table}
	
	\subsection{Organization}
	In Section \ref{secalg}, we propose two parameter-free single-loop algorithms, PF-AGP-NSC and PF-AGP-NC, for the nonconvex-strongly concave and nonconvex-concave minimax problems, and analyze their corresponding gradient complexities in two different settings. In Section \ref{secalg2}, we propose a parameter-free alternating gradient projection (PF-AGP-NL) algorithm for the nonconvex-linear minimax problem and analyze its corresponding gradient complexity. We report some numerical results in Section \ref{senu} and draw some conclusions in the last section.

	\noindent{\bfseries Notation}.
	For vectors, we use $\|\cdot\|$ to denote the $l_2$ norm. For a function $f(x, y): \mathbb{R}^n \times \mathbb{R}^m \rightarrow \mathbb{R}$, we use $\nabla_x f(x, y)$ (or $\nabla_y f(x, y)$) to denote the partial gradient of $f$ with respect to the first variable (or second variable) at the point $(x, y)$. Let $\mathcal{P}_\mathcal{X}$ and $\mathcal{P}_\mathcal{Y}$ denote the projections onto the sets $\mathcal{X}$ and $\mathcal{Y}$, respectively. Finally, we use the notation $\mathcal{O}(\cdot)$ to hide only absolute constants that do not depend on any problem parameters, and $\tilde{\mathcal{O}}(\cdot)$ to hide only absolute constants and logarithmic factors.
	A continuously differentiable function $f(\cdot)$ is called $\theta$-strongly convex if there exists a constant $\theta > 0$ such that for any $x, y \in \mathcal{X}$,
	\begin{equation}
		f(y) \geq f(x) + \langle \nabla f(x), y - x \rangle + \frac{\theta}{2}\|y - x\|^2. \label{sc}
	\end{equation}
	$f(\cdot)$ is called $\theta$-strongly concave if $-f$ satisfies (\ref{sc}).
	The symbolic function is represented as $\mbox{sgn}(x)$, that is:
	\begin{equation}
		\mbox{sgn}(x) =
		\begin{cases}
			1 & \text{if } x > 0, \\
			-1 & \text{if } x \leq 0.
		\end{cases}
	\end{equation}

	\section{\ {Completely parameter-free single-loop algorithms for nonconvex- (strongly) concave setting}}\label{secalg}
	In this section, we propose two completely parameter-free single-loop algorithms for solving the nonconvex-strongly concave and the nonconvex-concave minimax problems \eqref{p} respectively. Both algorithms do not require prior knowledge of Lipschitz constants, and the algorithm under nonconvex-strong concave settings also does not require prior knowledge of strong concavity coefficient. The two proposed algorithms are based on the AGP algorithm \citep{xu2023unified}, which uses the gradient of a regularized version of the original function, i.e.,
	\begin{equation*}
		f_k(x,y) = f(x,y) - \frac{c_k}{2}\|y\|^2,
	\end{equation*}
	where $c_k \geq 0$ is a regularization parameter. At the $k$-th iteration, AGP algorithm consists of the following two gradient projection steps for updating both $x$ and $y$:
	\begin{align*}
		x_{k+1}&= \mathcal{P}_\mathcal{X} \left( x_k - \tfrac{1}{\beta_k} \nabla _xf ( x_k,y_k ) \right),\\
		y_{k+1}&= \mathcal{P}_\mathcal{Y}\left( y_k+\tfrac{1}{\gamma_{k}} \nabla _y {f}( x_{k+1},y_k ) - \frac{1}{\gamma_{k}} c_{k} y_k \right),
	\end{align*}
	where  $\mathcal{P}_\mathcal{X}$ and $\mathcal{P}_\mathcal{Y}$ is the projection operator onto $\mathcal{X}$ and  $\mathcal{Y}$, respectively, and  $\beta_k > 0$, $\gamma_k >0$ are stepsize parameters.

	Although the AGP algorithm achieves the optimal iteration complexity dependence on $\varepsilon$ among single loop algorithms for nonconvex-(strongly) concave minimax problems, it requires the knowledge of the Lipschtiz constant $L$ and the strongly concave coefficient $\mu$, or at least knowledge of the Lipschtiz constant $L$  to calculate the stepsizes $\beta_k$ and $\gamma_k$ under the nonconvex-strongly concave setting and the nonconvex-concave setting, respectively, which limits its applicability. To eliminate the need for the prior knowledge of the Lipschtiz constant $L$ or the strongly concave coefficient $\mu$, we propose two parameter-free AGP algorithms by employing backtracking strategies for solving the nonconvex-strongly concave and the nonconvex-concave minimax problems, respectively. Similar backtracking strategies have been used in
	\citep{lan2023optimal,Liu2022,Nesterov2006,Nesterov2015}.

	\subsection{Nonconvex-strongly Concave Setting}
	We first propose a completely parameter-free AGP algorithm for solving nonconvex-strongly concave minimax problems.
	More specifically, in order to estimate the Lipschitz constants for $\nabla_x f(\cdot,y)$, $\nabla_y f(\cdot,y)$, $\nabla_y f(x,\cdot)$ denoted as $L_{11}$, $L_{12}$ and $L_{22}$, respectively, and the strongly concavity modulus  $\mu$, at each iteration of the proposed algorithm, we aim to find a tuple $(l_{11}^{k,i}, l_{12}^{k,i}, l_{22}^{k,i}, \mu_{k,i})$ by backtracking such that the following conditions for $(x_{k,i}, y_{k,i})$ are satisfied:
	\begin{align}
		C_1^{k,i} =& f(x_{k,i}, y_k) - f(x_k, y_k) - \langle \nabla_x f(x_k, y_k), x_{k,i} - x_k \rangle - \frac{l^{k,i}_{11}}{2} \|x_{k,i} - x_k\|^2 \leq 0\label{C1},\tag{C1} \nonumber\\
		C_2^{k,i} =& \|\nabla_y f(x_{k,i}, y_k) - \nabla_y f(x_k, y_k)\| - l_{12}^{k,i} \|x_{k,i} - x_k\| \leq 0,\label{C2}\tag{C2}\\
		C_3^{k,i} =&l_{22}^{k,i} \langle\nabla_y f(x_{k,i}, y_{k,i}) - \nabla_y f(x_{k,i}, y_{k,i}),y_{k,i}-y_k\rangle  \nonumber \\
		&+ \|\nabla_y f(x_{k,i}, y_{k,i}) - \nabla_y f(x_{k,i}, y_k)\|^2 \leq 0,\label{C3}\tag{C3}\\
		C_4^{k,i} =& \langle \nabla_y f(x_{k,i}, y_{k,i}) - \nabla_y f(x_{k,i}, y_k), y_{k,i} - y_k \rangle + \mu_{k,i} \|y_{k,i} - y_k\|^2 \leq 0.\label{C4}\tag{C4}
	\end{align}
	Otherwise, we enlarge $l_{11}^{k,i}, l_{12}^{k,i}, l_{22}^{k,i}$ by a factor of $2$ or reduce $\mu_{k,i}$ by half according to the sign of $C_1^{k,i}$, $C_2^{k,i}$, $C_3^{k,i}$ and $C_4^{k,i}$, respectively. We then update $\beta_k$ and $\gamma_{k}$ by the new estimates, i.e., $l_{11}^{k,i}, l_{12}^{k,i}, l_{22}^{k,i}, \mu_{k,i}$.
	The proposed algorithm for solving nonconvex-strongly concave problems, denoted as {PF-AGP-NSC},  is formally presented in Algorithm \ref{al_nc_sc}.
	
	\begin{algorithm}[t]
		\caption{A parameter-free alternating gradient projection ({PF-AGP-NSC}) algorithm for nonconvex-strongly concave minimax problems}
		\begin{algorithmic}
			\STATE \textbf{Step 1:} Input  $x_{1}, y_{1}$, $\beta_{0}$, $\gamma_{0}$, $l_{11}^{0}$, $l_{12}^{0}$, $l_{22}^{0}$, $\mu_0$; Set $k=1$.
			\STATE \textbf{Step 2:} {\bf Update $x_k$ and $y_k$}:
			\STATE\quad  \textbf{(a):}  Set $i=0$, $l_{11}^{k,i} = {l_{11}^{k-1}}{}$, $l_{12}^{k,i} = {l_{12}^{k-1}}$, $l_{22}^{k,i} = l_{22}^{k-1}$, $\mu_{k,i} = \mu_{k-1}$, $\beta_{k,i} = \beta_{k-1}$, $\gamma_{k,i} = \gamma_{k-1}$.
			\STATE\quad  \textbf{(b):}  Update $x_{{k,i}}$ and $y_{{k,i}}$: 
			\begin{align}
				x_{{k,i}} &=  \mathcal{P}_\mathcal{X}\left(x_{k} - \frac{1}{\beta_{k,i}} \nabla_x f(x_{k}, y_{k}) \right),\label{x_update}\\
				y_{{k,i}} &=  \mathcal{P}_\mathcal{Y}\left(y_{k} + \frac{1}{\gamma_{k,i}} \nabla_y f(x_{{k,i}}, y_{k})\right) \label{y_update}.
			\end{align}
			\STATE   \quad \textbf{(c):} Compute $ C_1^{k,i}$, $ C_2^{k,i}$, $ C_3^{k,i}$, $ C_4^{k,i}$ as in \eqref{C1}, \eqref{C2}, \eqref{C3}, \eqref{C4}, respectively;
			\STATE   \quad \textbf{(d):} Update $l^{k,i}_{11}$,  $l^{k,i}_{12}$,  $l^{k,i}_{22}$, $\mu_{k,i}$:
			\begin{align}
				l^{k,i+1}_{11}&=\frac{\mbox{sgn} (C_1^{k,i}) +3 }{2}  l^{k,i}_{11}, \quad 	l^{k,i+1}_{12}=\frac{\mbox{sgn} (C_2^{k,i}) +3 }{2}  l^{k,i}_{12},\\
				l^{k,i+1}_{22}&=\frac{\mbox{sgn} (C_3^{k,i}) +3 }{2}  l^{k,i}_{22}, \quad \mu_{k,i+1}=\frac{2}{\mbox{sgn} (C_4^{k,i}) +3 }  \mu_{k,i}.
			\end{align}
			\STATE   \quad \textbf{(e):} \textbf{If} $ C_1^{k,i}\leq 0$, $ C_2^{k,i}\leq 0$, $ C_3^{k,i}\leq 0$ and $ C_4^{k,i}\leq 0$,  \textbf{then} 
			\STATE  \qquad  \quad   $x_{k+1} = x_{k,i}$, $y_{k+1} = y_{k,i}$, $l_{11}^{k}=l_{11}^{k,i+1}$, $l_{12}^{k}=l_{12}^{k,i+1}$, $l_{22}^{k}=l_{22}^{k,i+1}$, $\mu_k = \mu_{k,i+1}$,  
			\STATE  \qquad  \quad$\beta_{k} = \beta_{k,i}$, $\gamma_{k} = \gamma_{k,i}$, go to \textbf{Step 3};\\ 	\STATE  \qquad  \quad  \textbf{Otherwise}, 
			
			$i = i+1$,  \ {$$\beta_{k,i}  ={l_{11}^{k,i}} + l_{12}^{k,i}  +  \frac{32(l_{12}^{k,i})^2(l^{k-1}_{12}+l^{k-1}_{22})}{\mu_{k,i}\mu_{k-1}},\quad \gamma_{k,i} = l_{12}^{k,i}+l_{22}^{k,i} ,$$ }
			\STATE  \qquad  \quad  go to \textbf{Step 2(b)}.
			\STATE \textbf{Step 3:} \textbf{If} some stationary condition is satisfied, \textbf{stop}; \textbf{Otherwise}, set $k = k + 1$, go to \textbf{Step 2}.
		\end{algorithmic}
		\label{al_nc_sc}
	\end{algorithm}
	
	It is worth noting that the parameter-free NeAda \citep{Yang2022} and TiAda \citep{Li2022} algorithms with adaptive step size are only used to solve the minimax problems when $x$ is unconstrained, while the  proposed {PF-AGP-NSC} algorithm can solve the constrained case. Compared with the SGDA-B \citep{Xu2024} algorithm and AGDA+ \citep{Zhang2024} algorithm using the backtracking framework under the nonconvex-strongly concave setting,  SGDA-B still requires knowledge of the strongly concavity modulus  $\mu$, while the proposed {PF-AGP-NSC} algorithm is a completely parameter-free algorithm that does not require any prior knowledge of the parameters. Moreover, the complexity of the  proposed PF-AGP-NSC is better than AGDA+.

	\subsubsection{Complexity analysis}\label{secom}
	In this subsection, we analyze the iteration complexity of Algorithm \ref{al_nc_sc} for solving nonconvex-strongly concave minimax optimization problems \eqref{p}, i.e., $f(x, y) $ is nonconvex with respect to $ x $ for any fixed $ y \in \mathcal{Y} $, and $ \mu $-strongly concave with respect to $ y $ for any given $ x \in \mathcal{X} $. We first need to make the following assumption about the smoothness of $ f(x, y) $.
	\begin{assumption}\label{as}
		$ f(x, y) $ has Lipschitz continuous gradients, i.e., there exist positive scalars $L_{11}, L_{12}, L_{22}$ such that for any $ x, \tilde{x} \in \mathcal{X}, y, \tilde{y} \in \mathcal{Y} $,
		\begin{align*}
			\| \nabla_x f(x, y) - \nabla_x f(\tilde{x}, y) \| &\leq L_{11} \| x - \tilde{x} \|, \\
			\| \nabla_y f(x, y) - \nabla_y f(\tilde{x}, y) \| &\leq L_{12} \| x - \tilde{x} \|,\\
			\| \nabla_y f(x, \tilde{y}) - \nabla_y f(x, y) \| &\leq L_{22} \| \tilde{y} - y \|. 
		\end{align*}
	\end{assumption}
	We denote  ${L} = \max\{L_{11}, L_{12}, L_{22}\}$. To analyze the convergence of Algorithm \ref{al_nc_sc}, we define the stationarity gap as the termination criterion as follows.
	\begin{definition}
		At each iteration of Algorithm \ref{al_nc_sc}, the stationarity gap for problem (\ref{p}) with respect to $ f(x, y) $ is defined as:
		\begin{equation}
			\nabla G(x_k, y_k) = \begin{bmatrix}
				\beta_k \left(x_k - \mathcal{P}_\mathcal{X} \left(x_k - \frac{1}{\beta_k} \nabla_x f(x_k, y_k)\right)\right) \\
				\gamma_k \left(y_k - \mathcal{P}_\mathcal{Y} \left(y_k + \frac{1}{\gamma_k} \nabla_y f(x_k, y_k)\right)\right)
			\end{bmatrix}.
		\end{equation}
		We denote $\nabla {G}_k= \nabla {G}(x_k, y_k)$,
		$(\nabla G_k)_{x}= \beta_k( x_k-\operatorname{\mathcal{P}_\mathcal{X}}( x_k-\frac{1}{\beta_k}\nabla _{x}f( x_k,y_k)))$,
		and $(\nabla G_k)_{y}=\gamma_k( y_k-\operatorname{\mathcal{P}_\mathcal{Y}}( y_k+ \frac{1}{\gamma_k} \nabla _{y}f( x_k,y_k )) )$.
		\label{gap}
	\end{definition}
	\begin{definition}
		At each iteration of Algorithm \ref{al_nc_sc}, the stationarity gap for problem (\ref{p}) with respect to $f_k(x, y)$ is defined as:
		$$
		\nabla \tilde{G}(x_k, y_k) = 
		\begin{bmatrix}
			\beta_k\left(x_k - P_\mathcal{X}\left(x_k - \frac{1}{\beta_k} \nabla_x f_k(x_k, y_k)\right)\right) \\
			\gamma_k \left(y_k - P_\mathcal{Y}\left(y_k + \frac{1}{\gamma_k} \nabla_y f_k(x_k, y_k)\right)\right) 
		\end{bmatrix}.
		$$ \label{d21}
	\end{definition}
	Under Assumption \ref{as}, we first establish the following lemma, which provides bounds on the changes in the function value when $x_k$ is updated at each iteration of Algorithm \ref{al_nc_sc}. This proof is similar to that of Lemma 2.1 in \citep{xu2023unified}, for the sake of completeness, we give its proof.

	\begin{lemma}
		\label{lemma:estimating_function_value_changes} Suppose that Assumption \ref{as} holds. 
		Let $\{(x_k, y_k)\}$ be a sequence generated by Algorithm \ref{al_nc_sc}, then we have
		\begin{equation}
			f(x_{k+1}, y_k) - f(x_k, y_k) \leq - \left(\beta_k - \frac{l^k_{11}}{2}\right) \|x_{k+1} - x_k\|^2 \label{es1}.
		\end{equation}
		\label{lem}
	\end{lemma}

	\begin{proof}
		By the optimality condition for \eqref{x_update}, we have
		\begin{equation}\label{lem:2.1:3}
			\langle  \nabla _xf\left( x_k,y_k \right) +\beta_k \left( x_{k+1}-x_k \right) ,x_k-x_{k+1} \rangle  \ge 0.
		\end{equation}
		The backtracking strategy \eqref{C1} implies that
		\begin{equation}
			f\left( x_{k+1},y_k \right) -f\left( x_k,y_k \right)\leq \langle  \nabla _xf\left( x_k,y_k \right) ,x_{k+1}-x_k \rangle  +\tfrac{l^k_{11}}{2}\| x_{k+1}-x_k \| ^2. \label{lem:2.1:4}
		\end{equation}
		By adding \eqref{lem:2.1:3} and \eqref{lem:2.1:4}, we obtain
		\begin{align*}
			f\left( x_{k+1},y_k \right) -f\left( x_k,y_k \right) &\le -\left( \beta_k -\tfrac{l^k_{11}}{2} \right) \| x_{k+1}-x_k \| ^2,
		\end{align*}
		which completes the proof. 
	\end{proof}


	
	Next, we provide an upper bound of the difference between $f(x_{k+1}, y_{k+1})$ and $f(x_k, y_k)$. 
	First, we need to make the following assumption on the parameter $\gamma_{k}$.
	\begin{assumption}
		$\{\gamma_k\}$ is a nonnegative monotonically increasing sequence. \label{a2}
	\end{assumption}
	\begin{lemma}
		Suppose that Assumptions \ref{as} and \ref{a2} hold. 
		Let $\{(x_k, y_k)\}$ be a sequence generated by Algorithm \ref{al_nc_sc}, \ {if $\gamma_k\ge l_{22}^k$}, then we have
		\begin{align}
			&f(x_{k+1}, y_{k+1}) - f(x_k, y_k) \nonumber\\
			\leq & -\left(\beta_k - \frac{l_{11}^k}{2} - \frac{l_{12}^{k}}{2}\right)\|x_{k+1} - x_k\|^2+\left(\frac{l^k_{12}}{2}+\frac{\gamma_{k-1}}{2}\right)\|y_{k+1} - y_k\|^2\nonumber \\ 
			&  -\left(\frac{\mu_{k-1}}{2} - \frac{\gamma_{k-1}}{2} \right)\|y_k - y_{k-1}\|^2  \label{34}.
		\end{align} \label{31}
	\end{lemma} 
	
	\begin{proof}
		The optimality condition for $y_k$ in \eqref{y_update} implies that
		\begin{equation}
			\langle \nabla_y f(x_k, y_{k-1}) - \gamma_{k-1} (y_k - y_{k-1}), y_{k+1} - y_k \rangle \leq 0. \label{37}
		\end{equation}
		By the concavity of $f(x,y)$ with respect to $y$, and combining \eqref{37}, we have 
		\begin{align}
			&f(x_{k+1}, y_{k+1}) - f(x_{k+1}, y_k) \nonumber\\
			\leq & \langle \nabla_y f(x_{k+1}, y_k), y_{k+1} - y_k \rangle\nonumber\\
			\leq & \langle \nabla_y f(x_{k+1}, y_k) - \nabla_y f(x_k, y_{k-1}), y_{k+1} - y_k \rangle+ \gamma_{k-1}\langle y_k-y_{k-1}, y_{k+1} - y_k \rangle.\label{key}
		\end{align}
		Denoting $ v_{k+1} = (y_{k+1} - y_k) - (y_k - y_{k-1}) $, we can write the first inner product term in the r.h.s. of (\ref{key}) as
		\begin{align}
			&\langle \nabla_y f(x_{k+1}, y_k) - \nabla_y f(x_k, y_{k-1}), y_{k+1} - y_k \rangle  \nonumber\\
			=& \langle \nabla_y f(x_{k+1}, y_k) - \nabla_y f(x_k, y_k), y_{k+1} - y_k \rangle \nonumber\\
			& + \langle \nabla_y f(x_k, y_k) - \nabla_y f(x_k, y_{k-1}), v_{k+1} \rangle\nonumber\\
			& + \langle \nabla_y f(x_k, y_k) - \nabla_y f(x_k, y_{k-1}), y_k - y_{k-1} \rangle. \label{three}
		\end{align}
		Next, we estimate the three terms on the right-hand side of (\ref{three}) respectively. By the backtracking strategy \eqref{C2} and the Cauchy-Schwarz inequality, we have

		\begin{equation}
			\langle \nabla_y f(x_{k+1}, y_k) - \nabla_y f(x_k, y_k), y_{k+1} - y_k \rangle \leq \frac{l_{12}^k}{2} \|x_{k+1} - x_k\|^2 + \frac{l^k_{12}}{2} \|y_{k+1} - y_k\|^2 \label{C2_1}
		\end{equation}
		and
		\begin{equation}
			\langle \nabla_y f(x_k, y_k) - \nabla_y f(x_k, y_{k-1}), v_{k+1} \rangle \leq \frac{1}{2 \gamma_{k-1}} \|f(x_{k}, y_k) - \nabla_y f(x_k, y_{k-1})\| + \frac{\gamma_{k-1}}{2} \|v_{k+1}\| \label{C3_1}.
		\end{equation}
		By the backtracking strategy \eqref{C3} and \eqref{C4}, we obtain
		\begin{equation}
			\begin{split}
				&\langle \nabla_y f(x_k, y_k) - \nabla_y f(x_k, y_{k-1}), y_k - y_{k-1} \rangle \\
				\leq & -\frac{1}{2 l_{22}^{k-1}} \|f(x_{k}, y_k) - \nabla_y f(x_k, y_{k-1})\| 
				-\frac{\mu_{k-1}}{2} \|y_k - y_{k-1}\|^2.
			\end{split}
		\end{equation}
		Moreover, it can be easily checked that
		\begin{equation}
			\langle y_k - y_{k-1}, y_{k+1} - y_k \rangle = \frac{1}{2} \|y_k - y_{k-1}\|^2 + \frac{1}{2} \|y_{k+1} - y_k\|^2 - \frac{1}{2} \|v_{k+1}\|^2. \label{th}
		\end{equation}
		Plugging \eqref{three}-\eqref{th} into (\ref{key}) and rearranging the terms, we conclude that
		
		\begin{align}
			&f(x_{k+1}, y_{k+1}) - f(x_{k+1}, y_k) \nonumber \\
			\leq &  \frac{l^k_{12}}{2} \|x_{k+1} - x_k\|^2 + \left(  \frac{l^k_{12}}{2} + \frac{\gamma_{k-1}}{2}\right) \|y_k - y_{k-1}\|^2 - \left(  \frac{\mu_{k-1}}{2} - \frac{\gamma_{k-1}}{2}\right) \|y_{k+1} - y_k\|^2. \label{314}
		\end{align}
		By Lemma \ref{lem}, we have
		\begin{equation}
			f(x_{k+1}, y_k) - f(x_k, y_k) \leq - (\beta_k - \frac{l_{11}^{k}}{2}) \|x_{k+1} - x_k\|^2. \label{315}
		\end{equation}
		The proof is then completed by combining (\ref{314}) with (\ref{315}).
	\end{proof}
	
	We now establish an important recursion for Algorithm \ref{al_nc_sc}.
	\begin{lemma}
		Suppose that Assumptions \ref{as} and \ref{a2} hold. 
		Let $\{(x_k, y_k)\}$ be a sequence generated by Algorithm \ref{al_nc_sc}, and denote
		\begin{align}
			f_{k+1}& = f(x_{k+1}, y_{k+1}), \quad S_{k+1} = \frac{4\gamma^2_k }{\mu_k} \|y_{k+1} - y_k\|^2, \nonumber\\
			F_{k+1} & = f_{k+1} + S_{k+1} -  \frac{7\gamma_k}{2}  \|y_{k+1} - y_k\|^2-\frac{4\gamma_k^2}{\mu_k}  \mathcal{D}^2_y.  \label{lem3-1}
		\end{align}
		If $\gamma_k\ge l_{22}^k$, then we have
		\begin{align}
			F_{k+1} - F_k
			\leq& -\left(  2 \gamma_k - \frac{l_{12}^k}{2} \right) \|y_{k+1}-y_k\|^2\nonumber\\
			& - \left( \beta_k - \frac{l_{11}^k}{2} - \frac{l_{12}^k}{2}- \frac{16(l_{12}^k)^2\gamma_{k-1}}{\mu_{k-1}\mu_k} \right) \|x_{k+1} - x_k\|^2,\label{316}
		\end{align}
		where $ \mathcal{D}_y = \max\{\|y_1 - y_2\| \mid y_1, y_2 \in \mathcal{Y}\} $. \label{32}
	\end{lemma}
	\begin{proof}
		By the optimality condition for $y_k$ in \eqref{y_update} we obtain
		\begin{equation}
			\langle \nabla_y f(x_{k+1}, y_{k}) - \gamma_{k} (y_{k+1} - y_{k}), y_{k} - y_{k+1} \rangle \leq 0. \label{36}
		\end{equation}
		Combining (\ref{36}) and (\ref{37}), we have
		\begin{align}
			&\langle \nabla_y f(x_k, y_{k-1}) - \nabla_y f(x_{k+1}, y_k)+\gamma_k( y_{k+1} - y_k)-\gamma_{k-1}( y_{k} - y_{k-1}), y_{k+1} - y_{k}\rangle \nonumber\\  
			&\leq  0,
		\end{align}
		which implies that
		\begin{align}
			&\gamma_k \|y_{k+1} - y_k\|^2 - \gamma_{k-1} \langle y_k-y_{k-1}, y_{k+1} - y_k \rangle\nonumber\\
			&\leq\langle \nabla_y f(x_{k+1}, y_k) - \nabla_y f(x_k, y_{k-1}), y_{k+1} - y_k \rangle.
		\end{align}
		Denoting $v_{k+1} = (y_{k+1}-y_k)-(y_k-y_{k-1})$, and using the identity $\frac{1}{2}\langle v_{k+1}, y_{k+1} - y_k\rangle = \frac{1}{2}\|y_{k+1} - y_k\|^2 + \frac{1}{2}\|v_{k+1}\|^2 - \frac{1}{2}\|y_k - y_{k-1}\|^2$, we conclude that
		\begin{align}
			&\gamma_k \|y_{k+1} - y_k\|^2 - \gamma_{k-1}\left( \frac{1}{2} \|y_{k+1} - y_{k}\|^2 + \frac{1}{2} \|y_k - y_{k-1}\|^2 - \frac{1}{2} \|v_{k+1}\|^2 \right)\nonumber  \\ &\quad\leq
			\langle \nabla_y f(x_{k+1}, y_k) - \nabla_y f(x_k, y_{k-1}), y_{k+1} - y_k \rangle.
		\end{align}
		By rearranging the terms of the above inequality, we then have
		\begin{align}
			&\left(\gamma_k - \frac{\gamma_{k-1}}{2}\right) \|y_{k+1} - y_k\|^2 + \frac{\gamma_{k-1}}{2} \|v_{k+1}\|^2 - \frac{\gamma_{k-1}}{2} \|y_k - y_{k-1}\|^2 \nonumber\\ 
			\leq &	\langle \nabla_y f(x_{k+1}, y_k) - \nabla_y f(x_k, y_{k-1}), y_{k+1} - y_k \rangle \nonumber \\
			=& \langle \nabla_y f(x_{k+1}, y_k) - \nabla_y f(x_k, y_k), y_{k+1} - y_k \rangle \nonumber \\
			&+\langle \nabla_y f(x_k, y_k) - \nabla_y f(x_k, y_{k-1}), v_{k+1} \rangle \nonumber\\
			& +\langle \nabla_y f(x_k, y_k) - \nabla_y f(x_k, y_{k-1}), y_k - y_{k-1} \rangle.  \label{Three}
		\end{align}
		By backtracking strategy \eqref{C2} and the Cauchy-Schwarz inequality, we can easily get that
		\begin{align}
			&\langle \nabla_y f(x_{k+1}, y_k) - \nabla_y f(x_k, y_k), y_{k+1} - y_k \rangle  \nonumber \\
			\leq & \frac{2(l_{12}^k)^2}{\mu_k} \|x_{k+1} - x_k\|^2 + \frac{\mu_k}{8} \|y_{k+1} - y_k\|^2.  \label{1}
		\end{align}
		Combining \eqref{C2_1}, \eqref{C3_1}, \eqref{Three} and \eqref{1}, we have
		\begin{align}
			& \frac{\gamma_k}{2} \|y_{k+1} - y_k\|^2 - \frac{\gamma_{k-1}}{2} \|y_k - y_{k-1}\|^2 \nonumber \\
			\leq & \frac{2(l_{12}^k)^2}{\mu_k} \|x_{k+1} - x_k\|^2 + \left(\frac{\mu_k}{8} + \frac{\gamma_{k-1}}{2}-\frac{\gamma_{k}}{2} \right) \|y_{k+1} - y_k\|^2 \nonumber\\
			& - \frac{\mu_{k-1}}{2} \|y_k - y_{k-1}\|^2.\label{ori}
		\end{align}
		By multiplying \ {$\frac{8\gamma_{k-1}}{\mu_{k-1}}$} on both sides of (\ref{ori}), using the definition of $S_{k+1}$ and Assumption \ref{a2}, we obtain
		\begin{equation}
			\begin{split}
				S_{k+1} - S_k 
				&\leq \frac{16(l_{12}^k)^2\gamma_{k-1}}{\mu_{k-1}\mu_k} \|x_{k+1} - x_k\|^2 + \gamma_{k-1} \|y_{k+1} - y_k\|^2  \\ \nonumber
				& \quad - {4\gamma_{k-1}}  \|y_k - y_{k-1}\|^2  + \left(\frac{4\gamma_k^2}{\mu_k} - \frac{4\gamma^2_{k-1}}{\mu_{k-1}}\right) \mathcal{D}_y^2.
			\end{split}
		\end{equation}
		The proof is then completed by Lemma \ref{31} and the definition.
	\end{proof}
	
	We are now ready to establish the iteration complexity for the {PF-AGP-NSC}  algorithm. In particular, letting $\nabla G(x_k, y_k)$ be defined as in Definition \ref{gap} and $\varepsilon > 0$ be a given target accuracy, we provide a bound on $T(\varepsilon)$, the first iteration index to achieve an $\varepsilon$-stationary point, i.e., $\|\nabla G(x_k, y_k)\| \leq \varepsilon$, which is equivalent to
	\begin{equation}
		T(\varepsilon) = \min\{k \mid \|\nabla G(x_k, y_k)\| \leq \varepsilon\}.
	\end{equation}
	\begin{theorem} \label{th1}
		Suppose that Assumptions \ref{as} and \ref{a2} hold. 
		Let $\{(x_k, y_k)\}$ be a sequence generated by Algorithm \ref{al_nc_sc}, {  $\beta_{k}  ={l_{11}^{k}} + l_{12}^{k}  +  \frac{32(l_{12}^{k})^2(l^{k-1}_{12}+l_{22}^{k-1})}{\mu_{k}\mu_{k-1}},  \gamma_{k} = l_{12}^{k}+l_{22}^{k}$},
		then it holds that
		\begin{equation}
			{T}(\varepsilon) \leq \frac{F_2 -\underline{F} }{d_1 \varepsilon^2} +1,
		\end{equation}
		where  $F_2$ is defined in \eqref{lem3-1}, 
		$\underline{F} = \underline{f} - \left( 14L_{12}+\frac{128L_{12}}{\mu} \right) \mathcal{D}^2_y
		$,
		with $ \mathcal{D}_y = \max\{\|y_1 - y_2\| \mid y_1, y_2 \in \mathcal{Y}\} $, $d_1 = \min\{\frac{3}{16l^1_{12}},\frac{\beta_{\text{min}}}{8L^2_{12}},\frac{1}{4\beta_{\max}}\}$, $\underline{f} = \min_{(x, y) \in X \times Y} f(x, y) $, $\beta_{\text{min}}={l^1_{11}} + {l^1_{12}} + \frac{32 (l_{12}^1)^2(l^1_{11}+l^1_{22})}{\mu_1^2 }$ and  $\beta_{\text{max}} = 2L_{11}+ 2{ L_{12}} + \frac{ 2^9L_{12}^2(L_{12}+L_{22})}{\mu^2}$.
	\end{theorem}	
	\begin{proof}
		It follows immediately from Definition \ref{gap} and \eqref{x_update} that
		\begin{equation}
			\| (\nabla G_k)_x \| = \beta_k \| x_{k+1} - x_k \| \label{323}.
		\end{equation}
		On the other hand, by  the triangle inequality, we obtain that
		\begin{align}
			\| (\nabla G_k)_y \| 
			&\leq \gamma_k \left\| P_y \left(y_k + \frac{1}{\gamma_k} \nabla_y f(x_{k+1}, y_k)\right) - P_y \left(y_k + \frac{1}{\gamma_k} \nabla_y f(x_k, y_k)\right) \right\|\nonumber \\
			&\quad + \gamma_k \| y_{k+1} - y_k \|\nonumber \\
			&\leq L_{12} \| x_{k+1} - x_k \| + \gamma_k \| y_{k+1} - y_k \|.  \label{324}
		\end{align}
		By combining (\ref{323}) and (\ref{324}), and using the Cauchy-Schwarz inequality, we obtain
		\begin{align}
			\| \nabla G_k \|^2 &\leq (\beta_k^2 + 2 L_{12}^2) \| x_{k+1} - x_k \|^2 + 2 \gamma_k^2 \| y_{k+1} - y_k \|^2.
			\label{325}
		\end{align}
		Let $  d^{(k)}_1= {\min \left\{ \frac{2\gamma_k-\frac{l^k_{12}}{2}}{2 \gamma_{k}^2}, \frac{\beta_{k} - \frac{l^k_{11}}{2}-\frac{  l^k_{12}}{2} - \frac{16 (l^k_{12})^2 \gamma_{k-1}}{\mu_{k-1}\mu_k}}{\beta_{k}^2 + 2 L_{12}^2} \right\}}>0  $. Multiplying both sides of (\ref{325}) by $ d^{(k)}_1 $, and using (\ref{316}) in Lemma \ref{32}, we have
		\begin{align}
			d^{(k)}_1 \| \nabla G_k \|^2 & \leq F_k - F_{k+1}.
		\end{align}
		By using the definitions of $d_1$, $\beta_k$ and $\gamma_k$, we have
		$d_1<d^{(k)}_1$, then we get
		\begin{align}
			d_1 \| \nabla G_k \|^2 
			\leq F_k - F_{k+1}.
		\end{align}
		Summing up the above inequalities from $ k = 2 $ to $ k = T(\varepsilon) $, we obtain
		\begin{align}
			\sum_{k=2}^{T(\varepsilon)} d_1 \|\nabla G_k\|^2 &\leq F_2 - F_{T(\varepsilon)+1}  .
			\label{327}
		\end{align}
		Note that by the definition of $F_{k+1} $ in Lemma \ref{32}, we have
		\begin{align}
			F_{T(\varepsilon)+1} &= f_{T(\varepsilon)+1} + S_{T(\varepsilon)+1} -\frac{2\gamma_{T(\varepsilon)}^2}{\mu_{T(\varepsilon)}}   \mathcal{D}^2_y -  \frac{7}{2} \gamma_{T(\varepsilon)+1}  ||y_{T(\varepsilon)+1} - y_{T(\varepsilon)}||^2 \nonumber\\
			&\geq \underline{f} - \left( 14L_{12}+\frac{128L_{12}}{\mu} \right) \mathcal{D}^2_y =\underline{F},
		\end{align}
		where the inequality follows from the definitions of $\underline{f}$ and $\mathcal{D}_y$, and the fact that $S_k \geq 0$ ($\forall k \geq 1$). 
		We then conclude from (\ref{327}) that
		\begin{align}
			\sum_{k=2}^{T(\varepsilon)} d_1 \|\nabla G_k\|^2 &\leq {F_2} - F_{T(\varepsilon)+1} \leq {F}_2 - \underline{F}  ,
		\end{align}
		which, in view of the definition of $T(\varepsilon)$, implies that
		\begin{equation}
			d_1 (T(\varepsilon)-1) \varepsilon^2 \leq F_2 - \underline{F} ,
		\end{equation}
		or equivalently,
		\begin{equation}
			T(\varepsilon) \leq \frac{F_2 -  \underline{F}}{d_1 \varepsilon^2}+1.
		\end{equation}
		We complete the proof.
	\end{proof}
	\begin{remark} 
		In PF-AGP-NSC algorithm, the \ {total} number of backtracking step is upper bounded by \ {$\sum\limits_{k=2}^{T(\varepsilon)} (\log_2\left(\frac{l^k_{11}}{l_{11}^{k-1}}\right)+\log_2\left(\frac{l^{k}_{12}}{l_{12}^{k-1}}\right)+\log_2\left(\frac{l^{k}_{22}}{l_{22}^{k-1}}\right)+\log_2\left(\frac{\mu_{k-1}}{\mu_{k}}\right))=\log_2\left(\frac{2L_{11}}{l_{11}^{1}}\right)+\log_2\left(\frac{2L_{12}}{l_{12}^{1}}\right)+\log_2\left(\frac{2L_{22}}{l_{22}^{1}}\right)+\log_2\left(\frac{2\mu_1}{\mu}\right)=O({\log(\kappa}))$}.
		According to  Theorem \ref{th1}, \ { $d_1 = O(\frac{1}{L\kappa^2})$, $\underline{F}= O(L\kappa)$ }, the total number of  gradient calls of  PF-AGP-NSC to obtain an $\varepsilon$-stationary point  that satisfies $\|\nabla G(x_k,y_k)\|\leq \varepsilon$ is upper bounded by \ {$\mathcal{O} ( L^2\kappa^3\varepsilon^{-2})$} under the nonconvex-strongly concave setting, where $\kappa$ is the condition number.
	\end{remark}
	\begin{remark}
		If additional information is available, the iteration complexity of the algorithm can be further improved. To illustrate this point, if $\mu$ and $\underline{S}$ are known in advance, in Appendix \ref{app1} we propose a restarted variant of the PF-AGP-NSC algorithm, named \text{rPF-AGP-NSC}, and show that the gradient complexity of the algorithm can be improved to $\mathcal{O}\left(\log(L) L\kappa^2\varepsilon^{-2}\right)$, which is better than SGDA-B.
	\end{remark}
	
	\subsection{Nonconvex-concave Setting}
	\begin{algorithm}[t]
		\caption{A parameter-free alternating gradient projection ({PF-AGP-NC}) algorithm for nonconvex-concave minimax problems}
		\begin{algorithmic}
			\STATE \textbf{Step 1:} Input  $x_{1}, y_{1}$, $\beta_{0}$, $\gamma_{0}$, $c_{0}$, $l_{11}^{0}$, $l_{12}^{0}$, $l_{22}^{0}$; Set $k=1$.
			\STATE \textbf{Step 2:} {\bf Updata $x_k$ and $y_k$}:
			\STATE\quad  \textbf{(a):}  Set $i=0$,\ { $l_{11}^{k,i} = {l_{11}^{k-1}}$, $l_{12}^{k,i} = {l_{12}^{k-1}}$}, $l_{22}^{k,i} = l_{22}^{k-1}$, $\beta_{k,i} = \beta_{k-1}$, $\gamma_{k,i} = \gamma_{k-1}$, $c_{k,i} =c_{k-1}$.
			\STATE\quad  \textbf{(b):}  Update $x_{{k,i}}$ and $y_{{k,i}}$: 
			\begin{align}
				x_{{k,i}} &=  \mathcal{P}_\mathcal{X}\left(x_{k} - \frac{1}{\beta_{k,i}} \nabla_x f(x_{k}, y_{k}) \right),\\
				y_{{k,i}} &=  \mathcal{P}_\mathcal{Y}\left(y_{k} + \frac{1}{\gamma_{k,i}} \nabla_y f(x_{{k,i}}, y_{k})-\frac{1}{\gamma_{k,i}}{c_{k,i}}y_{k}\right) \label{y_update2}.
			\end{align}
			\STATE   \quad \textbf{(c):} \ {Compute $ C_1^{k,i}$, $ C_2^{k,i}$, $ C_5^{k,i}$, as in \eqref{C1}, \eqref{C2}, \eqref{NC3}};
			\STATE   \quad \textbf{(d):} Update $l^{k,i}_{11}$,  $l^{k,i}_{12}$,  $l^{k,i}_{22}$:
			\begin{align}
				l^{k,i+1}_{11}&=\frac{\mbox{sgn} (C_1^{k,i}) +3 }{2}  l^{k,i}_{11},\quad
				l^{k,i+1}_{12}=\frac{\mbox{sgn} (C_2^{k,i}) +3 }{2}  l^{k,i}_{12},\\l^{k,i+1}_{22}&=\frac{\mbox{sgn} (C_5^{k,i}) +3 }{2}  l^{k,i}_{22}.
			\end{align}
			\STATE   \quad \textbf{(e):}\ { \textbf{If} $ C_1^{k,i}\leq 0$, $ C_2^{k,i}\leq 0$ and $ C_5^{k,i}\leq 0$},  \textbf{then} 
			\STATE  \qquad  \quad   $x_{k+1} = x_{k,i}$, $y_{k+1} = y_{k,i}$, $l_{11}^{k}=l_{11}^{k,i+1}$, $l_{12}^{k}=l_{12}^{k,i+1}$, $l_{22}^{k}=l_{22}^{k,i+1}$,  $\beta_{k} = \beta_{k,i}$,
			\STATE  \qquad  \quad $\gamma_{k} = \gamma_{k,i}$,  $c_k = c_{k,i}$, go to \textbf{Step 3};\\ 	\STATE  \qquad  \quad  \textbf{Otherwise}, $i = i+1$, 
			\begin{align}
				\beta_{k,i} = \frac{l_{12}^{k,i}}{20l^{k-1,i}_{22}} + \frac{2(l_{12}^{k,i})^2 \sqrt{k}}{l_{12}^{k-1,i}} 
				,\quad \gamma_{k,i} =20l^{k,i}_{22}, \quad c_{k,i} = \frac{19l^{k,i}_{22}}{k^{1/4}},
			\end{align}
			\STATE  \qquad  \quad go to \textbf{Step 2(b)}.
			\STATE \textbf{Step 3:} \textbf{If} some stationary condition is satisfied, \textbf{stop}; \textbf{Otherwise}, set $k = k + 1$, go to \textbf{Step 2}.
		\end{algorithmic}
		\label{al_nc_c}
	\end{algorithm} 
	In this subsection, we propose a completely parameter-free AGP algorithm for solving nonconvex-concave minimax problems. Similar to the idea of the {PF-AGP-NSC} Algorithm, we estimate the Lipschitz constants $L_{11}$, $L_{12}$ and $L_{22}$ iteratively, however we do not need to estimate the strongly concavity modulus  $\mu$.  At each iteration of the proposed algorithm, we aim to find a tuple $(l^{k,i}_{11},l^{k,i}_{12})$ by backtracking such that 
	conditions \eqref{C1}, \eqref{C2}.
	The proposed algorithm, denoted as {PF-AGP-NC}, is formally presented in Algorithm \ref{al_nc_c}.
	
	Please note that the existing SGDA-B algorithm \citep{Xu2024} for solving nonconvex-concave minimax problems requires the realization of the diameter of the feasible set $\mathcal{Y}$ in order to construct the key auxiliary problem, while the proposed {PF-AGP-NC} algorithm is a completely parameter-free algorithm that does not require any prior parameters.
	\subsubsection{Complexity Analysis}
	In this subsection, we analyze the iteration complexity of Algorithm \ref{al_nc_c} for solving (\ref{p}) under the general nonconvex-concave setting. 
	Similar to the idea of the {PF-AGP-NSC} Algorithm, we estimate the Lipschitz constants $L_{11}$, $L_{12}$ and $L_{22}$ iteratively, however we do not need to estimate the strongly concavity modulus  $\mu$. At each iteration of the proposed algorithm, we aim to find a tuple $(l^{k,i}_{11},l^{k,i}_{12},l^{k,i}_{22})$ by backtracking such that 
	conditions \eqref{C1}, \eqref{C2} and the following condition \eqref{NC3} are satisfied:
	\begin{align}
		{C}_5^{k,i} =& (l_{22}^{k,i}+c_{k,i}) \langle\nabla_y f_{k,i}(x_{k,i}, y_{k,i}) - \nabla_y f_{k,i}(x_{k,i}, y_{k}),y_{k,i}-y_k\rangle \nonumber \\ &+\|\nabla_y f_{k,i}(x_{k,i}, y_{k,i}) - \nabla_y f_{k,i}(x_{k,i}, y_k)\|^2 \leq 0 \label{NC3}\tag{C5}
	\end{align}
	The backtracking strategy \eqref{NC3} is a key inequality that we will use in the following two lemmas to establish some important recursions of the {PF-AGP-NC} algorithm in the nonconvex-concave setting. This is also one of the main differences between proofs in the nonconvex-strongly concave setting and proofs in the nonconvex-concave setting.
	
	\begin{lemma}\label{lm33}
		Suppose that Assumptions \ref{as} and \ref{a2} hold. Let $\{(x_k, y_k)\}$ be a sequence generated by Algorithm \ref{al_nc_c}. If $\forall k$, $\beta_k > l^k_{11}$ and \ {$\gamma_{k} \geq {{l}_{22}^{k} + c_{k}}
			$}, then
		\begin{align}
			& f({x_{k+1}},y_{k+1}) - f({x_k},y_k) \nonumber\\ 
			\leq & - \left(   \frac{{\beta}_k}{2} - \frac{\left({l_{12}^{k}}\right)^2}{2\gamma_{k-1}} \right) \|x_{k+1} - x_k\|^2 +{\gamma_{k-1}} \|y_{k+1}- y_{k}\|^2 + \frac{\gamma_{k-1}}{2} \|y_k - y_{k-1}\|^2\nonumber\\
			&\quad + \frac{c_{k-1}}{2} \left( \|y_{k+1}\|^2 - \|y_k\|^2 \right)+ \frac{c_k}{2} \|y_{k+1}-y_k\|^2.\label{kk}
		\end{align}
	\end{lemma}
	\begin{proof}
		The optimality condition for $y_k$ in \eqref{y_update2} implies that
		\begin{equation}
			\langle \nabla_y f_{k-1} (x_k, y_k) - \gamma_{k-1} (y_k - y_{k-1}), y_k - y_{k+1} \rangle \leq 0.\label{334}
		\end{equation}
		%
		The concavity of $f_k(x_{k+1}, y)$ with respect to $y$ together with (\ref{334}) then imply that
		\begin{align}
			& f_k (x_{k+1}, y_{k+1}) - f_k (x_{k+1}, y_{k}) \nonumber\\
			\leq & \langle \nabla_y f_k (x_{k+1}, y_{k})- \nabla_y f_{k-1} (x_{k}, y_{k}), y_{k+1}-y_k \rangle\nonumber\\
			& + \langle \nabla_y f_{k-1} (x_{k}, y_{k-1}), y_{k+1} - y_k \rangle\nonumber\\
			&= \langle \nabla_y f_k (x_{k+1}, y_{k}) - \nabla_y f_{k-1} (x_{k}, y_{k}), y_{k+1} - y_k \rangle\nonumber\\
			& + \langle \nabla_y f_{k-1} (x_{k}, y_{k}) - \nabla_y f_{k-1} (x_{k}, y_{k-1}), v_{k+1} \rangle\nonumber\\
			& + \langle \nabla_y f_{k-1} (x_{k}, y_{k}) - \nabla_y f_{k-1} (x_{k}, y_{k-1}), y_{k} - y_{k-1} \rangle  \nonumber\\
			& + \gamma_{k-1} \langle y_k - y_{k-1}, y_{k+1} - y_k \rangle, \label{335}
		\end{align}
		where $ v_{k+1} = y_{k+1} - y_k - (y_k - y_{k-1}) $. We now provide upper bounds for the inner product terms on the right hand side of \eqref{335}. Firstly, by the definition of $ f_k(x_{k+1}, y_k) $ and $ f_{k-1}(x_k, y_k) $, the backtracking strategy \eqref{C2}  and the Cauchy-Schwarz inequality, we have
		\begin{align}
			&\langle \nabla_y f_k (x_{k+1}, y_k) - \nabla_y f_{k-1} (x_k, y_k), y_{k+1} - y_k \rangle\nonumber\\ 
			&= \langle \nabla_y f(x_{k+1}, y_k) - \nabla_y f (x_k, y_k), y_{k+1} - y_k \rangle  - (c_k - c_{k-1})\langle y_k, y_{k+1} - y_k \rangle \nonumber\\
			&\leq \frac{(l_{12}^k)^2}{2\gamma_{k-1}} \|x_{k+1} - x_k\|^2 + \frac{\gamma_{k-1}}{2} \|y_{k+1} - y_k\|^2 - \frac{c_k - c_{k-1}}{2} (\|y_{k+1}\|^2 - \|y_k\|^2)\nonumber\\
			&\quad + \frac{c_k - c_{k-1}}{2} \|y_{k+1}-y_k\|^2\nonumber\\
			&\leq \frac{(l_{12}^k)^2}{2\gamma_{k-1}} \|x_{k+1} - x_k\|^2 + \frac{\gamma_{k-1}}{2} \|y_{k+1} - y_k\|^2 - \frac{c_k - c_{k-1}}{2} (\|y_{k+1}\|^2 - \|y_k\|^2) \nonumber\\
			&\quad + \frac{c_k}{2} \|y_{k+1}-y_k\|^2. \label{336}
		\end{align}
		Secondly, by the Cauchy-Schwarz inequality, we get
		\begin{align}
			&\langle \nabla_y f_{k-1} (x_k, y_k) - \nabla_y f_{k-1} (x_k, y_{k-1}), v_{k+1} \rangle\nonumber \\ 
			\leq& \frac{1}{2 \gamma_{k-1}} \| \nabla_y f_{k-1} (x_k, y_k) - \nabla_y f_{k-1} (x_k, y_{k-1}) \|^2 + \frac{\gamma_{k-1}}{2} \| v_{k+1} \|^2.
			\label{337}
		\end{align}
		Thirdly, by \eqref{NC3} , we have that
		\begin{align}
			&\langle \nabla_y f_{k-1}(x_k, y_k) - \nabla_y f_{k-1}(x_k, y_{k-1}), y_k - y_{k-1} \rangle\nonumber\\
			\leq & -\frac{1}{{l}_{22}^{k-1} + c_{k-1}} \| \nabla_y f_{k-1}(x_k, y_k) - \nabla_y f_{k-1}(x_k, y_{k-1}) \|^2
			.
			\label{338}
		\end{align}
		We also have the following equation that
		\begin{align}
			&\gamma_{k-1}\langle y_{k} - y_{k-1}, y_{k+1} - y_k \rangle \nonumber\\
			=& \frac{\gamma_{k-1}}{2} \| y_{k+1} - y_k \|^2 + \frac{\gamma_{k-1}}{2} \| y_k - y_{k-1} \|^2 - \frac{\gamma_{k-1}}{2} \| v_{k+1} \|^2. \label{339}
		\end{align}
		By $\gamma_{k} \geq {{l}_{22}^{k} + c_{k}}$, plugging \eqref{336}-\eqref{339} into (\ref{335}), 
		and using the definition of $f_k(x_{k+1}, y_{k+1})$ and $f_k(x_{k+1}, y_k)$,   we obtain
		\begin{align}
			&f(x_{k+1}, y_{k+1}) - f(x_{k+1}, y_k) \nonumber\\
			\leq& \frac{(l_{12}^{k})^2}{2 \gamma_{k-1}} \| x_{k+1} - x_k \|^2 + \gamma_{k-1} \| y_{k+1} - y_k \|^2 + \frac{\gamma_{k-1}}{2}  \| y_k - y_{k-1} \|^2\nonumber\\
			&\quad + \frac{c_{k-1}}{2} \left( \|y_{k+1}\|^2 - \|y_k\|^2 \right)+ \frac{c_k}{2} \|y_{k+1}-y_k\|^2  \label{340}.
		\end{align}
		By  $  \beta_k > l^k_{11}$, we have
		\begin{equation}
			f(x_{k+1}, y_{k}) - f(x_k, y_k) \leq -\frac{  \beta_k }{2}\| x_{k+1} - x_k \|^2 \label{341}.
		\end{equation}
		By adding (\ref{340}) and (\ref{341}), we complete the proof.
	\end{proof}
	\begin{lemma}\label{lm34}
		Suppose that Assumptions \ref{as} and \ref{a2} hold. Let $\{(x_k, y_k)\}$ be a sequence generated by Algorithm \ref{al_nc_c}. Denote
		\begin{align}
			S_{k+1} &= \frac{8\gamma_k^2}{c_{k}} \| y_{k+1} - y_k \|^2 + {8}{\gamma_k} \left( \frac{c_{k+1}}{c_{k}}-1 \right) \| y_{k+1} \|^2,\\
			F_{k+1} &= {f(x_{k+1}, y_{k+1})} + S_{k+1} - \frac{15\gamma_k}{2} \|y_{k+1} - y_k\|^2 - \frac{c_k}{2} \|y_{k+1}\|^2.
		\end{align}
		If
		\begin{equation}
			\beta_k > l^k_{11}, \quad \frac{\gamma_{k+1}}{c_{k+1}} - \frac{\gamma_k}{c_k} \leq \frac{1}{5},\quad {c_{k+1}} - {c_k} \leq {\gamma_{k+1}-\gamma_k}, \quad \gamma_k \geq {{l}^k_{22} + c_k} ,\label{342}
		\end{equation}
		then $\forall k \geq 2$,
		\begin{align}
			F_{k+1} - F_k =& - \frac{19\gamma_k}{10}\|y_{k+1} - y_k\|^2 -\left(\frac{  \beta_k}{2} - \frac{(l^k_{12})^2}{2\gamma_{k-1}}-\frac{16(l^k_{12})^2\gamma_{k-1}}{c_{k-1}^2}\right)\|x_{k+1} - x_k\|^2 \nonumber\\
			& +8\left( {\gamma_{k}} \frac{c_{k+1}}{c_k} - {\gamma_{k-1}} \frac{c_k}{c_{k-1}} \right) \|y_{k+1}\|^2+\frac{c_{k-1}-c_k}{2}\|y_{k+1}\|^2.
		\end{align}
	\end{lemma}
	\begin{proof}
		The optimality condition for $y_k$ in \eqref{y_update2} implies that
		\begin{equation}
			\langle \nabla_y f_k(x_{k+1}, y_k)-\gamma_k(y_{k+1}-y_k), y_{k} - y_{k+1} \rangle \leq 0\label{333},
		\end{equation}
		By (\ref{334}) and (\ref{333}), we have  
		\begin{align}
			&\langle \nabla_y f_{k-1}(x_{k}, y_{k-1}) - \nabla_y f_{k}(x_{k+1}, y_{k}) +\gamma_k( y_{k+1} - y_{k})  - \gamma_{k-1} (y_{k} - y_{k-1}), y_{k+1} - y_{k} \rangle \nonumber\\ 
			&\leq 0,\nonumber
		\end{align}
		and
		\begin{align}
			&	\gamma_{k} \| y_{k+1} - y_{k} \|^2 - \gamma_{k-1} \langle y_{k+1} - y_{k}, y_{k} - y_{k-1} \rangle\nonumber\\
			\leq& \langle \nabla_y f_{k}(x_{k+1}, y_{k}) - \nabla_y f_{k-1}(x_{k}, y_{k-1}), y_{k+1} - y_{k} \rangle,\label{344}
		\end{align}
		which can be rewritten as
		\begin{align}
			&\left(\gamma_k - \frac{\gamma_{k-1}}{2}\right) \|y_{k+1} - y_k\|^2 + \frac{\gamma_{k-1}}{2} \|v_{k+1}\|^2 - \frac{\gamma_{k-1}}{2} \|y_k - y_{k-1}\|^2\nonumber\\
			\leq & \langle \nabla_y f_k(x_{k+1}, y_k) - \nabla_y f_{k-1}(x_k, y_{k-1}), y_{k+1} - y_k \rangle\nonumber\\
			=& \langle \nabla_y f_k(x_{k+1}, y_k) - \nabla_y f_{k-1}(x_k, y_k), y_{k+1} - y_k \rangle\nonumber\\
			& + \langle \nabla_y f_{k-1}(x_k, y_k) - \nabla_y f_{k-1}(x_k, y_{k-1}), v_{k+1}  \rangle\nonumber\\
			&+ \langle \nabla_y f_k(x_{k-1}, y_{k-1}) - \nabla_y f_{k-1}(x_k, y_{k-1}), y_{k} - y_{k-1} \rangle.
		\end{align}
		Using an argument similar to the proof of (\ref{336})-(\ref{339}), by \eqref{NC3}, $f_k(x,y)$ is concave with respect to $y$ and the Cauchy-Schwarz inequality, we conclude from the above inequality that
		\begin{align}
			&\left(\gamma_k - \frac{\gamma_{k-1}}{2}\right) \|y_{k+1} - y_k\|^2 + \frac{\gamma_{k-1}}{2} \|v_{k+1}\|^2 - \frac{\gamma_{k-1}}{2} \|y_k - y_{k-1}\|^2\nonumber\\
			\leq & \; \frac{(l_{12}^k)^2}{2a_k}  \|x_{k+1} - x_k\|^2 + \frac{a_k}{2} \|y_{k+1} - y_k\|^2 - \frac{c_k - c_{k-1}}{2} (\|y_{k+1}\|^2 - \|y_k\|^2)\nonumber\\
			& + \frac{c_k - c_{k-1}}{2} \|y_{k+1} - y_k\|^2 + \frac{1}{2\gamma_{k-1}} \|\nabla_y f_{k-1}(x_k, y_k) - \nabla_y f_{k-1}(x_k, y_{k-1})\|^2  \nonumber\\ 
			& +\frac{\gamma_{k-1}}{2}\|v_{k+1}\|- \frac{1}{{l}_{22}^{k-1} + 2c_{k-1}} \|\nabla_y f_{k}(x_k, y_k) - \nabla_y f_{k}(x_k, y_{k-1})\|^2  \nonumber\\ 
			&  - \frac{c_{k-1} }{2} \|y_k - y_{k-1}\|^2,\label{345}
		\end{align}
		for any $a_k > 0$. 
		Combining $\gamma_k \geq {{l}^k_{22} + c_k}$ and rearranging the terms in (\ref{345}), we obtain
		\begin{align}
			&\frac{\gamma_k}{2} \|y_{k+1} - y_k\|^2 + \frac{c_k - c_{k-1}}{2} \|y_{k+1}\|^2 \nonumber\\
			&\quad\leq \frac{\gamma_{k-1}}{2} \|y_k - y_{k-1}\|^2 + \left( \frac{a_k}{2} + \frac{\gamma_{k-1}}{2} - \frac{\gamma_k}{2}+\frac{c_k}{2}-\frac{c_{k-1}}{2} \right) \|y_{k+1} - y_k\|^2  + \frac{c_k - c_{k-1}}{2} \|y_{k}\|^2 \nonumber\\
			&\quad\quad+ \frac{ (l^k_{12})^2 }{2a_k} \|x_{k+1} - x_{k}\|^2 - \frac{c_{k-1}}{2} \|y_k - y_{k-1}\|^2 \nonumber\\
			&\quad\leq \frac{\gamma_{k-1}}{2} \|y_k - y_{k-1}\|^2 +  \frac{a_k}{2}\|y_{k+1} - y_k\|^2  + \frac{c_k - c_{k-1}}{2} \|y_{k}\|^2 + \frac{ (l^k_{12})^2 }{2a_k} \|x_{k+1} - x_{k}\|^2 				
			\\&\quad\quad- \frac{c_{k-1}}{2} \|y_k - y_{k-1}\|^2.
		\end{align}
		By multiplying $\frac{16\gamma_{k-1}}{c_{k-1}}$ on both sides of the above inequality, setting $a_k = \frac{c_k}{2}$, and using the definition of $S_{k+1}$, (\ref{342}) and $\{\gamma_k\}$ is nodecreasing, we then obtain
		\begin{align}
			S_{k+1} - S_k &\leq 8\left( {\gamma_{k}} \frac{c_{k+1}}{c_k} - {\gamma_{k-1}} \frac{c_k}{c_{k-1}} \right) \|y_{k+1}\|^2 + \frac{16 (l^k_{12})^2 \gamma_{k-1}}{c_{k-1}^2} \|x_{k+1} - x_k\|^2 \nonumber\\
			&\quad+ \frac{28\gamma_{k}}{5} \|y_{k+1} - y_k\|^2 - {8}{\gamma_{k-1}} \|y_{k} - y_{k-1}\|^2.
			\label{347}
		\end{align}
		%
		Combining  (\ref{kk}) and (\ref{347}), and using the definition of $F_{k+1}$,  we complete the proof.
	\end{proof}
	
	We are now ready to establish the iteration complexity for the {PF-AGP-NC} algorithm to achieve an $\varepsilon$-stationary point under the general nonconvex–concave setting.
	\begin{theorem}\label{th2}
		Suppose that Assumptions \ref{as} and \ref{a2} hold. Let $\{(x_k, y_k)\}$ be a sequence generated by Algorithm \ref{al_nc_c}. If $  \beta_k = \frac{l_{12}^k}{20l^{k-1}_{22}} + \frac{2(l_{12}^k)^2 \sqrt{k}}{l_{22}^{k-1}} 
		$, $\gamma_{k} =20l^k_{22}$ ,  $c_k = \frac{19l^k_{22}}{k^{1/4}}$,  then for any given $\varepsilon > 0$, $k \geq 2$, we have
		$$
		T(\varepsilon) \leq   \max \left(  \left( \frac{ \hat{\alpha} d_2d_3}{ \varepsilon^2}+1\right)^2,  \frac{76^4L_{22}^{4}\hat{\mathcal{D}}_y^4}{\varepsilon^{4}} \right),
		$$
		where  $d_2= F_2 - \underline{F} +339L_{22}  \hat{\mathcal{D}}_y^2$, $d_3 = \max\{\bar{d}_1,\max\limits_{{k=2,\cdots,T(\varepsilon)}} \{\frac{19^2\cdot20^2 l_{22}^kl_{22}^{k-1}}{16(l_{12}^k)^2\}}\}\} $, $\hat{\mathcal{D}}_y = \max\{\|y\| \, |  y \in Y\}$ with $\underline{F} = {\underline{f}} - 639L_{22}\hat{\mathcal{D}}_y^2$, $\underline{f} = \min_{(x,y) \in X \times Y} f(x, y)$,  $ \hat{\alpha} =  \max\limits_{{k=2,\cdots,T(\varepsilon)}}\{\frac{2^8\cdot5^2(l^k_{12})^2}{19^2\cdot20l_{22}^{k-1}}\}$ and
		$\bar{d}_1 = \max\limits_{{k=2,\cdots,T(\varepsilon)}} \left\{ 4 + \frac{19^4}{2^{15} \cdot 5^4} + \frac{19^4 \cdot 20^2(l_{22}^k)^2}{2^{15} \cdot 5^4 \cdot (l_{12}^k)^2} \right\}$.
	\end{theorem}
	\begin{proof}
		We can easily see that the relations in \eqref{342} are satisfied by the settings of $c_k$, $\gamma_k$ and $\beta_k$. Denote $\alpha_k = \frac{16 (l_{12}^k)^2 \gamma_{k-1}}{c_{k-1}^2}$. Then it follows from the selection of  $  \beta_k$ and $\alpha_k$ that
		\begin{equation}
			\frac{  \beta_k}{2} - \frac{(l^k_{12})^2}{2\gamma_{k-1}}-\frac{16(l^k_{12})^2\gamma_{k-1}}{c_{k-1}^2} = \alpha_k .
		\end{equation}
		This observation, in view of Lemma \ref{lm33}, implies that
		\begin{align}
			&{\alpha_k} \left\| x_{k+1} - x_k \right\|^2 + \frac{19\gamma_k}{10} \left\| y_{k+1} - y_k \right\|^2 \nonumber\\
			\leq& F_k -F_{k+1} 
			+8\left( {\gamma_{k}} \frac{c_{k+1}}{c_k} - {\gamma_{k-1}} \frac{c_k}{c_{k-1}} \right) \|y_{k+1}\|^2 +\frac{c_{k-1}-c_k}{2}\|y_{k+1}\|^2. \label{348}
		\end{align}
		We can easily check from the definition of $ f_k(x, y) $ that
		\begin{equation}
			\| \nabla G_k \| - \| \tilde{\nabla} G_k \| \leq c_k \| y_k \|. 
		\end{equation}
		By replacing $f$ with $f_k$, similar to (\ref{323}) and (\ref{324}), we immediately obtain that
		\begin{equation}
			\|(\tilde{\nabla} G_k)_x \| =   \beta_k    \| x_{k+1} - x_k \|, \label{349}
		\end{equation}
		and
		\begin{equation}
			\| (\tilde{\nabla} G_k)_y \| \leq \gamma_k \| y_{k+1} - y_k \| + l^k_{12} \| x_{k+1} - x_k \|. \label{350}
		\end{equation}
		Combining (\ref{349}) and (\ref{350}), and using the Cauchy-Schwarz inequality, we have
		\begin{equation}
			\| \tilde{\nabla} G_k \|^2 \leq (  \beta_k ^2 + 2(l^k_{12})^2) \| x_{k+1} - x_k \|^2 + 2{{\gamma_k}^2} \| y_{k+1} - y_k \|^2. \label{351}
		\end{equation}
		Since both $\alpha_k$ and $  \beta_k$ are in the same order when $k$ becomes large enough, it then follows from the definition of $\bar{d}_1$ that $\forall k \geq 2$,
		\begin{align}
			\frac{(  \beta_k )^2 + 2({l^k_{12}})^2}{\alpha_k^2} &= 
			\frac{\left(\frac{{l^k_{12}}}{\gamma_{k-1}} + \frac{64 ({l}_{12}^k)^2\gamma_{k-1}}{c_{k-1}^2}   \right)^2 + 2({l}^k_{12})^2}{\alpha_k^2}\nonumber\\ 
			&\leq  \max\limits_{k=2,3\cdots T(\varepsilon)} \{4+ \frac{19^4}{2^{15}\cdot 5^4 }+\frac{19^4\gamma^2_{k}}{2^{15}\cdot5^4({l}_{12}^k)^2} \}= \bar{d}_1\label{352}.
		\end{align}
		Combining the previous two inequalities in (\ref{351}) and (\ref{352}), we obtain
		\begin{equation}
			\| \nabla {\tilde{G}_k}  \|^2 \leq  {\bar{d}_1(\alpha_k)^2}  \| x_{k+1} - x_k \|^2 + 2 \gamma_k^2 \| y_{k+1} - y_k \|^2 \label{353}.
		\end{equation}
		Denote \ {$d_k^{(2)} = \frac{1}{\max\{ \bar{d}_1 \alpha_k, {\frac{19}{20} \gamma_k} \}}$}, $\tilde{c}_k = \frac{19}{20 k^{1/4}}$, by multiplying $d^{(2)}_k$ on the both sides of (\ref{353}), and using (\ref{348}), we have
		\begin{align}
			&d_k^{(2)} \| \nabla \tilde{G}_k \|^2 \nonumber\\
			\leq& F_k - F_{k+1}  +8\left( {\gamma_{k}} \frac{c_{k+1}}{c_k} - {\gamma_{k-1}} \frac{c_k}{c_{k-1}} \right) \|y_{k+1}\|^2 
			+\frac{c_{k-1}-c_k}{2}\|y_{k+1}\|^2\nonumber\\
			\leq& F_k - F_{k+1}  +8\left( {\gamma_{k}} \frac{c_{k+1}}{c_k} - {\gamma_{k-1}} \frac{c_k}{c_{k-1}} \right) \|y_{k+1}\|^2 
			+40L_{22}\frac{\tilde{c}_{k-1}-\tilde{c}_k}{2}\|y_{k+1}\|^2.\label{354}
		\end{align}
		Denoting
		\begin{equation}
			\tilde{T}(\varepsilon) = \min \left\{ k \middle| \left\| \nabla \tilde{G}(x_k, y_k) \right\| \leq \frac{\varepsilon}{2}, k \geq 2 \right\},	
		\end{equation}
		summing both sides of (\ref{354}) from $ k = 2 $ to $ k = \tilde{T}(\varepsilon) $, we then obtain
		\begin{align}
			& \sum_{k=2}^{\tilde{T}{(\varepsilon)}} d^{(2)}_k \left\| \nabla \tilde{G}_k \right\|^2 \nonumber\\
			\leq &  F_2 - F_{\tilde{T}(\varepsilon)}  + 8 \left(  \frac{\gamma_{\tilde{T}(\varepsilon)-1}c_{\tilde{T}(\varepsilon)}}{c_{\tilde{T}(\varepsilon)-1}}-\frac{\gamma_{1}c_2}{ c_1}  \right) \hat{\mathcal{D}}_y^2 + 40L_{22}\left( \frac{\tilde{c}_1}{2} - \frac{\tilde{c}_{\tilde{T}(\varepsilon)}}{2} \right) \hat{\mathcal{D}}^2_y \nonumber\\
			\leq & F_2 -  F_{\tilde{T}(\varepsilon)} + 339L_{22}\hat{\mathcal{D}}_y^2.
			\label{355}
		\end{align}
		Note that by the definition of $F_{k+1}$ in Lemma \ref{lm34}, we have
		\begin{equation}
			\begin{split}
				F_{\tilde{T}(\varepsilon)} &\geq {\underline{f}} -  639L_{22}\hat{\mathcal{D}}_y^2 = \underline{F},
			\end{split}
		\end{equation}
		where $ \underline{f} = \min_{(x, y) \in X \times Y} f(x, y) $. We then conclude from (\ref{355}) that
		\begin{equation}
			\sum_{k=2}^{\infty} d^{(2)}_k \left\| \nabla \tilde{G}_k \right\|^2 \leq F_2 - \underline{F} +339L_{22} \hat{\mathcal{D}}_y^2= d_2. \label{356}
		\end{equation}
		We can see from the selection of  $d_3$ that \ {$ d_3 \geq \max\{\bar{d}_1, \frac{19\gamma_{k}}{20\alpha_k}\} $}, which implies $d^{(2)}_k \geq \frac{1}{d_3\alpha_k}$,  by multiplying $d_3$  on the both sides of (\ref{356}), and combining the definition of  $d_2$,  we have  $\sum_{k=2}^{\infty} \frac{1}{\alpha_k} \left\| \nabla \tilde{G}_k \right\|^2 \leq d_2d_3$,
		which, by the definition of $ \tilde{T}(\varepsilon) $, implies that
		\begin{equation}
			\frac{\varepsilon^2}{4} \leq \frac{d_2d_3}{\sum_{k=2}^{ \tilde{T}{(\varepsilon)} }\frac{1}{\alpha_k}} \leq \frac{d_2d_3}{\sum_{k=2}^{ \tilde{T}{(\varepsilon)} }\frac{1}{\bar{\alpha}}}. \label{357}
		\end{equation}
		Where $\bar{\alpha}= {(k-1)^{1/2}}\hat{\alpha}$, by using the fact $ \sum_{k=2}^{ \tilde{T}{(\varepsilon)}} \frac{1}{\sqrt{k}} \geq \sqrt{\tilde{T}{(\varepsilon)}} - 1$  and (\ref{357}), we conclude\ { $	\frac{\varepsilon^2}{4} \leq \frac{\hat\alpha d_2d_3}{ \sqrt{\tilde{T}(\varepsilon)} -1}$} or equivalently,
		\begin{equation}
			\tilde{T}(\varepsilon)  \leq \left( \frac{ \hat{\alpha} d_2d_3}{ \varepsilon^2}+1\right)^2.
		\end{equation}
		On the other hand, if  $k \geq \frac{76^4L_{22}^{4}\hat{\mathcal{D}}_y^4}{\varepsilon^{4}}$, then  $c_k = \frac{19\gamma_k}{20 k^{1/4}} \leq \frac{\varepsilon}{2\hat{{\mathcal{D}_y}}}$. This inequality together with the definition of  $\hat{\mathcal{D}}_y$, then imply that $c_k \left\| y_k \right\| \leq \frac{\varepsilon}{2}$. Therefore, there exists a
		\begin{align}
			T(\varepsilon) &\leq \max \left( \tilde{T}{(\varepsilon)}, \frac{76^4L_{22}^{4}\hat{\mathcal{D}}_y^4}{\varepsilon^{4}}\right)\nonumber\\
			&\leq \max \left(  \left( \left( \frac{ \hat{\alpha} d_2d_3}{ \varepsilon^2}+1\right)^2+1\right)^2,  \frac{76^4L_{22}^{4}\hat{\mathcal{D}}_y^4}{\varepsilon^{4}} \right),
		\end{align}
		such that $ \left\| \nabla G_k \right\| \leq \left\| \nabla \tilde{G}_k \right\| + c_k \left\| y_k \right\| \leq  \frac{\varepsilon}{2}  + \frac{\varepsilon}{2} = \varepsilon$. 
	\end{proof}
	
	%
	
	
	\begin{remark}
		In PF-AGP-NC algorithm, the total number of backtracking step is upper bounded by $\sum\limits_{k=2}^{T(\varepsilon)} \left(\log_2\left(\frac{l^k_{11}}{l_{11}^{k-1}}\right)+\log_2\left(\frac{l^{k}_{12}}{l_{12}^{k-1}}\right)+\log_2\left(\frac{l^{k}_{22}}{l_{22}^{k-1}}\right)\right)=\log_2\left(\frac{2L_{11}}{l_{11}^{1}}\right)+\log_2\left(\frac{2L_{12}}{l_{12}^{1}}\right)+\log_2\left(\frac{2L_{22}}{l_{22}^{1}}\right)=\mathcal{O}(\log(L))$.
		According to Theorem \ref{th2}, $d_2 = O(L)$, $\bar{d}_1=O(1)$, $d_3 = O(1)$, $\hat{\alpha} = O(\log(L)L)$, hence the total number of  gradient calls of  PF-AGP-NC to obtain an $\varepsilon$-stationary point  that satisfies $\|\nabla G(x_k,y_k)\|\leq \varepsilon$ is $\mathcal{O} ( \log^2(L) L^4\varepsilon^{-4})$ under the nonconvex-concave setting.
	\end{remark}
\begin{remark}
	In the nonconvex-concave setting, when $\varepsilon < \frac{L}{{2}^{\log^2(L) L}}$ holds, $\log^2( L) L^4 \varepsilon^{-4} < L^3 \log(L \varepsilon^{-1}) \varepsilon^{-4}$, which means that the total number of gradient calls for PF-AGP-NC algorithm, i.e., $\mathcal{O}(\log^2(L) L^4 \varepsilon^{-4})$, which is asymptotically smaller than the total number of gradient calls for SGDA-B algorithm, i.e., $\mathcal{O}(L^3 \log(L \varepsilon^{-1}) \varepsilon^{-4})$, especially when $\varepsilon \to 0^+$. On the other hand, SGDA-B algorithm requires additional prior information, notably $ \underline{S} $ and $ \mathcal{D}_y$.
	%
	%
\end{remark}

\section{A completely parameter-free single-loop algorithm for nonconvex-linear minimax problems}\label{secalg2}
In this section, we propose a completely parameter-free AGP algorithm, denoted as {PF-AGP-NL}, for solving nonconvex-linear minimax problems.
The proposed algorithm is similar to PF-AGP-NC, while utilizes the gradient of a regularized variant of the original function, i.e.,
\begin{equation*}
	\tilde{f}_k(x,y) = f(x,y) - \frac{c_k}{2}\|y\|^2-\frac{d_{k}}{2}\|y-y_k\|^2,
\end{equation*}
where $c_k \geq 0$ and $d_k \geq 0$ are regularization parameters.
At the $k$ -th iteration, the variable $x$ is updated by minimizing a linear approximation of $f(x, y_k)$ plus a regularization term, i.e., 
\begin {equation} 
\begin {aligned} x_{k+1} &= \arg \min _{x \in \mathcal {X}} \left \langle \nabla _x f(x_k, y_k), x - x_k \right \rangle + \frac { \beta _k}{2} \| x - x_k \| ^2 \\ &= \mathcal {P}_{ \mathcal {X}} \left ( x_k - \frac {1}{ \beta _k} \nabla _x f(x_k, y_k) \right ), 
\end {aligned} 
\end {equation} 
where $ \mathcal {P}_{ \mathcal {X}}$ denotes the projection operator and $ \beta _k > 0$ is the step-size parameter. Meanwhile, for the variable $y$ , the proposed algorithm solves the following inner maximization subproblem with respect to $y$ using a regularized  function instead of the original $f(x,y)$,
\begin {equation} 
y_{k+1} = \arg \max_{y \in \mathcal {Y}} \tilde {f}_k(x_{k+1}, y).
\end {equation}
Similar to the idea of PF-AGP-NSC and PF-AGP-NC algorithms, we iteratively estimate the Lipschitz constants $L_{11}$ and $L_{12}$. In each iteration of the algorithm, our goal is to find a tuple $(l^{k,i}_{11},l^{k,i}_{12})$ that satisfies conditions \eqref{C1} and \eqref{C2} by backtracking. The proposed algorithm, denoted as PF-AGP-NL, is formally presented in Algorithm \ref{al_nc_l}.
	\begin{algorithm}[t]
		\caption{A parameter-free alternating gradient projection ({PF-AGP-NL}) algorithm for nonconvex-linear minimax problems}
		\begin{algorithmic}
			\STATE \textbf{Step 1:} Input  $x_{1}, y_{1}$, $\beta_{0}$, $\gamma_{0}$, $c_{0}$, $d_{0}$, $l_{11}^{0}$, $l_{12}^{0}$; Set $k=1$.
			\STATE \textbf{Step 2:} {\bf Updata $x_k$ and $y_k$}:
			\STATE\quad  \textbf{(a):}  Set $i=0$,\ { $l_{11}^{k,i} = {l_{11}^{k-1}}$, $l_{12}^{k,i} = {l_{12}^{k-1}}$},  $\beta_{k,i} = \beta_{k-1}$, $c_{k,i} =c_{k-1}$, $d_{k,i} = d_{k-1}$.
			\STATE\quad  \textbf{(b):}  Update $x_{{k,i}}$ and $y_{{k,i}}$: 
			\begin{align}
				x_{{k,i}} &=  \mathcal{P}_\mathcal{X}\left(x_{k} - \frac{1}{\beta_{k,i}} \nabla_x f(x_{k}, y_{k}) \right),\\
				y_{{k,i}} & = \arg \max _{y \in \mathcal {Y}}f(x_{k,i},y) - \frac{c_{k,i}}{2}\|y\|^2-\frac{d_{k,i}}{2}\|y-y_k\|^2, \label{y_update3}.
			\end{align}
			\STATE   \quad \textbf{(c):} \ {Compute $ C_1^{k,i}$ and  $ C_2^{k,i}$, as in \eqref{C1} and  \eqref{C2}};
			\STATE   \quad \textbf{(d):} Update $l^{k,i}_{11}$ and $l^{k,i}_{12}$:
			\begin{align}
				l^{k,i+1}_{11}=\frac{\mbox{sgn} (C_1^{k,i}) +3 }{2}  l^{k,i}_{11},\quad
				l^{k,i+1}_{12}=\frac{\mbox{sgn} (C_2^{k,i}) +3 }{2}  l^{k,i}_{12},
			\end{align}
			\STATE   \quad \textbf{(e):}\ { \textbf{If} $ C_1^{k,i}\leq 0$ and $ C_2^{k,i}\leq 0$},  \textbf{then} 
			\STATE  \qquad  \quad   $x_{k+1} = x_{k,i}$, $y_{k+1} = y_{k,i}$, $l_{11}^{k}=l_{11}^{k,i+1}$, $l_{12}^{k}=l_{12}^{k,i+1}$,  $\beta_{k} = \beta_{k,i}$,
			$c_k = c_{k,i}$, $d_{k} = d_{k,i}$,\STATE  \qquad  \quad  go to \textbf{Step 3};\\ 	\STATE  \qquad  \quad  \textbf{Otherwise}, $i = i+1$, 
			\begin{align}
				\rho_{k,i}=	2 \max\{l_{11}^{k,i},l^{k,i}_{12}\}, \quad
				\beta_{k,i} = 2\rho_{k,i} k^{1/3} + \rho_{k,i}
				, \quad c_{k,i} = \frac{\rho_{k,i}}{k^{1/3}},\quad d_{k,i} =\frac{\rho_{k,i}}{16k^{1/3}}
			\end{align}
			\STATE  \qquad  \quad go to \textbf{Step 2(b)}.
			\STATE \textbf{Step 3:} \textbf{If} some stationary condition is satisfied, \textbf{stop}; \textbf{Otherwise}, set $k = k + 1$, go to \textbf{Step 2}.
		\end{algorithmic}
		\label{al_nc_l}
	\end{algorithm}
	\subsection{Complexity Analysis}
	In this subsection, we analyze the iteration complexity of Algorithm PF-AGP-NL for solving (\ref{p}) under the general nonconvex-linear setting. 
	\begin{lemma} 
		\label{lemma:descent_property}
		Under Assumption \ref{as}, let $\{(x_k, y_k)\}$ be the sequence generated by Algorithm \ref{al_nc_l}. Then we have
		\begin{equation}
			\begin{split}
				&\tilde{f}_{k+1}(x_{k+1}, y_{k+1}) - \tilde{f}_k(x_k, y_k) \\
				\leq & - \left(\beta_{k} - \frac{l_{11}^k}{2} - \frac{(l_{12}^k)2}{c_{k}}\right)\|x_{k+1} - x_k\|^2  \\
				& - (-\frac{c_{k}}{4}+\frac{c_{k-1}}{2}  +\frac{d_k-d_{k-1}}{2})\|y_{k+1} - y_k\|^2 + \frac{c_{k-1} - c_{k+1}}{2}\|y_{k+1}\|^2 \\
				& - \frac{c_{k-1} - c_{k}}{2}\|y_k\|^2+\frac{d_{k-1}}{2}\|y_k-y_{k-1}\|^2.
			\end{split}
		\end{equation}
	\end{lemma}
	\begin{proof}
		First, by the strong concavity of ${f}_k(x, y)$ with respect to $y$, we have
		\begin{equation}\label{eq:strong_concavity}
			\tilde{f}_k(x_{k+1}, y_{k+1}) - \tilde{f}_k(x_{k+1}, y_k) 
			\leq \langle \nabla_y \tilde{f}_k(x_{k+1}, y_k), y_{k+1} - y_k \rangle 
			- \frac{c_k+d_k}{2}\|y_{k+1} - y_k\|^2. 
		\end{equation}
		From the optimality condition of \eqref{y_update3} in Algorithm 1, we obtain
		\begin{equation}
			\langle \nabla_y \tilde{f}_{k-1}(x_k, y_k), y - y_k \rangle \leq 0.
		\end{equation}
		Let $y = y_{k+1}$ and we then get
		\begin{equation}
			\langle \nabla_y \tilde{f}_{k-1}(x_k, y_k), y_{k+1} - y_k \rangle \leq 0. \label{eq:optimal condition}
		\end{equation}
		Substituting \eqref{eq:optimal condition} into \eqref{eq:strong_concavity}, we get
		\begin{equation}
			\begin{split}
				\tilde{f}_k(x_{k+1}, y_{k+1}) - \tilde{f}_k(x_{k+1}, y_k) 
				\leq& \langle \nabla_y \tilde{f}_k(x_{k+1}, y_k) - \nabla_y \tilde{f}_{k-1}(x_k, y_k), y_{k+1} - y_k \rangle\\
				&- \frac{c_k+d_k}{2}\|y_{k+1} - y_k\|^2.   \label{eq:intermediate_inequality}
			\end{split}
		\end{equation}
		By the Cauchy-Schwarz inequality and the backtracking strategy \eqref{C2}, we obtain
		\begin{equation}
			\langle \nabla_y \tilde{f}(x_{k+1}, y_k) - \nabla_y \tilde{f}(x_k, y_k), y_{k+1} - y_k \rangle 
			\leq \frac{(l_{12}^k)^2}{c_{k}}\|x_{k+1} - x_k\|^2 + \frac{c_{k}}{4}\|y_{k+1} - y_k\|^2.  \label{eq:cauchy_schwarz}
		\end{equation}
		Moreover, we have
		\begin{equation}
			\begin{split}
				\langle y_k, y_{k+1} - y_k \rangle 
				= \frac{1}{2}\|y_{k+1}\|^2 - \frac{1}{2}\|y_k\|^2 - \frac{1}{2}\|y_{k+1} - y_k\|^2. \label{eq:norm_identity}
			\end{split}
		\end{equation}
		Combining \eqref{eq:cauchy_schwarz}-\eqref{eq:norm_identity} with \eqref{eq:intermediate_inequality} and applying the Cauchy-Schwarz inequality, we obtain
		\begin{equation}
			\begin{split}
				& \tilde{f}_k(x_{k+1}, y_{k+1}) - \tilde{f}_k(x_{k+1}, y_k) \\
				\leq& \frac{(l_{12}^k)^2}{c_{k}}\|x_{k+1} - x_k\|^2 - (-\frac{c_{k}}{4}+\frac{c_{k-1}}{2} +\frac{d_k-d_{k-1}}{2})\|y_{k+1} - y_k\|^2\\ 
				&+\frac{d_{k-1}}{2}\|y_k-y_{k-1}\| +\frac{c_{k-1}-c_k}{2}(\|y_{k+1}\|^2-\|y_{k}\|^2). \label{eq:combined_inequality}
			\end{split}
		\end{equation}
		From Lemma \ref{lemma:estimating_function_value_changes} and \eqref{eq:combined_inequality}, we obtain
		\begin{align}
			\tilde{f}_k(x_{k+1}, y_{k+1}) - \tilde{f}_k(x_k, y_k)
			\leq &- \left(\beta_{k} - \frac{l_{11}^k}{2} - \frac{(l_{12}^k)^2}{c_{k}}\right) \|x_{k+1} - x_k\|^2
			+\frac{d_{k-1}}{2}\|y_k-y_{k-1}\|\nonumber\\
			&- (-\frac{c_{k}}{4}+\frac{c_{k-1}}{2}  +\frac{d_k-d_{k-1}}{2})\|y_{k+1} - y_k\|^2\nonumber\\
			& + \frac{c_{k-1} - c_k}{2}\left(\|y_{k+1}\|^2 - \|y_k\|^2\right). \label{eq:lemma_application}
		\end{align}
		Finally, combining the definition of $f_k(x, y)$ with \eqref{eq:lemma_application}, we conclude that
		\begin{align}
			\tilde{f}_{k+1}(x_{k+1}, y_{k+1}) - \tilde{f}_k(x_k, y_k) 
			\leq &- \left(\beta_{k} - \frac{l_{11}^k}{2} - \frac{(l_{12}^k)^2}{c_{k}}\right)\|x_{k+1} - x_k\|^2 
			+\frac{d_{k-1}}{2}\|y_k-y_{k-1}\|\nonumber\\
			&+ \frac{c_{k-1} - c_{k+1}}{2}\|y_{k+1}\|^2 - \frac{c_{k-1} - c_k}{2} \|y_k\|^2\nonumber\\
			&- (-\frac{c_{k}}{4}+\frac{c_{k-1}}{2}  -\frac{d_{k-1}}{2})\|y_{k+1} - y_k\|^2,
		\end{align}
		which completes the proof.
	\end{proof}
	\begin{theorem} \label{thm:convergence_rate}
		Suppose Assumption \ref{as} holds. Set
		\begin{equation}
			H_{k+1} := \tilde{f}_{k+1}(x_{k+1}, y_{k+1}) +\frac{d_{k}}{2}\|y_{k+1}-y_k\|^2- \frac{c_k - c_{k+1}}{2} \|y_{k+1}\|^2.
		\end{equation}
		For any given $\varepsilon > 0$, choose
		\begin{equation}
			c_k = \frac{\rho_k}{k^{1/3}}, \quad d_k=\frac{c_k}{8}, \quad \beta_{k} = 2\rho_k k^{1/3} + \rho_k \label{eq:parameter_setting}.
		\end{equation}
		Let $\rho_k = 2 \max\{l_{11}^k,l^k_{12}\}$, then
		\begin{equation}
			T(\varepsilon) \leq \max \left\{ \left( \frac{H_2 - \underline{H} + 11 \hat{\mathcal{D}}_y^2}{\tau \varepsilon^2} + 1 \right)^{3/2}, \left(\frac{8L \hat{\mathcal{D}}_y}{\varepsilon}\right)^3 \right\},
		\end{equation}
		where $\tau = \frac{1}{2^6 3^2 L}$, $\underline{H} = \underline{F} - 4{L} \hat{\mathcal{D}}_y^2$, $\underline{F} = \min_{(x, y) \in \mathcal{X} \times \mathcal{Y}} \tilde{F}(x, y)$ and $\hat{\mathcal{D}}_y = \max\{\|y\| \, | y \in Y\}$.
	\end{theorem}
	\begin{proof}
		From Lemma \ref{lemma:descent_property} and the definition of $H_{k+1}$, we get
		\begin{equation}
			\begin{split}
				&\left(\beta_k - \frac{l_{11}^k}{2} - \frac{(l_{12}^k)^2}{c_{k}}\right) \|x_{k+1} - x_k\|^2 + (\frac{c_{k}}{4}-d_k)\|y_{k+1} - y_k\|^2 \\
				\leq &H_k - H_{k+1} + \frac{c_{k-1} - c_{k}}{2} \|y_{k+1}\|^2+\frac{c_{k}-c_{k-1}}{2}\|y_{k+1} - y_k\|^2+\frac{d_{k-1}-d_{k}}{2}\|y_{k+1} - y_k\|^2. \label{eq:descent_inequality}
			\end{split}
		\end{equation}
		By setting the parameters of $\beta_k$, we can easily obtain
		\begin{equation}
			\beta_k - \frac{l_{11}^k}{2} - \frac{(l_{12}^k)^2}{c_{k}} \geq \frac{\rho_k^2}{c_{k}} \quad \text{and} \quad \frac{c_{k}}{4}-d_k = \frac{c_k}{8}\label{eq:beta_bound}
		\end{equation}
		Combined with \eqref{eq:descent_inequality}-\eqref{eq:beta_bound}, we get
		\begin{equation}
			\begin{split}
				&\frac{\rho_k^2}{c_{k}} \|x_{k+1} - x_k\|^2 +\frac{c_{k}}{8} \|y_{k+1} - y_k\|^2\\ \leq& H_k - H_{k+1} + \frac{c_{k-1} - c_{k}}{2} \|y_{k+1}\|^2+\frac{c_{k}-c_{k-1}}{2}\|y_{k+1} - y_k\|^2+\frac{d_{k-1}-d_{k}}{2}\|y_{k+1} - y_k\|^2. \label{eq:simplified_inequality}
			\end{split}
		\end{equation}
		According to \eqref{x_update}, we immediately get
		\begin{equation}
			\|\left(\nabla \tilde{G}_k\right)_x\| = \beta_k\|x_{k+1} - x_k\|. \label{eq:gradient_x}
		\end{equation}
		On the other hand, according to \eqref{y_update} and the triangle inequality, we can prove that
		\begin{equation}
			\begin{split}
				\|\left(\nabla \tilde{G}_k\right)_y\| &\leq {\rho_k} \|y_{k} - A + A - B\| \\
				&\leq l_{12}^k\|x_{k+1} - x_k\| + \left(2\rho_k + c_k\right) \|y_{k+1} - y_k\|, \label{eq:gradient_y}
			\end{split}
		\end{equation}
		where $A = \mathcal{P}_y\left(y_{k+1} + \frac{1}{\rho_k} \nabla_y f_k(x_{k+1}, y_{k+1})\right)$ and $B = \mathcal{P}_y\left(y_k + \frac{1}{\rho_k} \nabla_y f_k(x_k, y_k)\right)$.
		Combining \eqref{eq:gradient_x}, \eqref{eq:gradient_y} and \eqref{C2}, and applying the Cauchy-Schwarz inequality and $c_k \leq \rho_k$, we get
		\begin{equation}
			\begin{split}
				\|\nabla \tilde{G}_k\|^2 &\leq \left(\beta_k^2 + 2(l_{12}^k)^2\right) \|x_{k+1} - x_k\|^2 + 2\left(2\rho_k + c_k\right)^2 \|y_{k+1} - y_k\|^2 \\
				&\leq \left(\beta_k^2 + 2(l_{12}^k)^2\right) \|x_{k+1} - x_k\|^2 + 18\rho_k^2 \|y_{k+1} - y_k\|^2. \label{eq: gradient_bound}
			\end{split}
		\end{equation}
		Multiplying both sides of \eqref{eq: gradient_bound} by $m_k c_{k}$, combined with \eqref{eq:simplified_inequality} and using the definition of $\hat{\mathcal{D}}_y$, we get
		\begin{equation}
			\begin{split}
				m_k c_{k} \|\nabla \tilde{G}_{k+1}\|^2 &\leq \frac{\rho_k^2}{c_{k}} \|x_{k+1} - x_k\|^2 + \frac{c_{k}}{8} \|y_{k+1} - y_k\|^2\\
				&\leq H_k - H_{k+1} + \frac{c_{k-1} - c_{k}}{2} \|y_{k+1}\|^2+\frac{c_{k}-c_{k-1}}{2}\|y_{k+1} - y_k\|^2\\&\quad+\frac{d_{k-1}-d_{k}}{2}\|y_{k+1} - y_k\|^2\\
				&\leq H_k - H_{k+1} + 4L({\bar{c}_{k-1} - \bar{c}_{k}}) \hat{\mathcal{D}}_y^2+\frac{7}{4}(\rho_{k}-\rho_{k-1})\hat{\mathcal{D}}_y^2, \label{eq:combined_bound}
			\end{split}
		\end{equation}
		where $m_k = \min\left\{\frac{\rho^2_{k}}{c^2_{k}(2(l_{12}^k)^2 + \beta^2_{k})}, \frac{1}{144\rho^2_k}\right\} = \frac{1}{144\rho^2_k}$ and $\bar{c}_k = \frac{1}{k^{1/3}}$.
		Define $\tilde{T}(\varepsilon) := \min \{k \mid \|\nabla \tilde{G}(x_k, y_k)\| \leq \frac{\varepsilon}{2}, k \geq 2\}$. Summing \eqref{eq:combined_bound} from $k = 2$ to $\tilde{T}(\varepsilon)$, and using $c_k \leq \rho_k$, we obtain
		\begin{equation}
			\sum_{k=2}^{\tilde{T}(\varepsilon)} m_k c_{k} \|\nabla \tilde{G}_{k+1}\|^2 \leq \sum_{k=2}^{\tilde{T}(\varepsilon)} \left(H_k - H_{k+1} +4L({\bar{c}_{k-1} - \bar{c}_{k}}) \hat{\mathcal{D}}_y^2+11(\rho_{k}-\rho_{k-1})\hat{\mathcal{D}}_y^2,\right).
		\end{equation}
		Further simplification yields
		\begin{equation}
			\sum_{k=2}^{\tilde{T}(\varepsilon)} m_k c_{k} \|\nabla \tilde{G}_{k+1}\|^2 \leq H_2 - H_{\tilde{T}(\varepsilon)+1} + 11L \hat{\mathcal{D}}_y^2. \label{eq:sum_inequality}
		\end{equation}
		From the definition of $H_{k+1}$, we have
		\begin{equation}
			\begin{split}
				H_{\tilde{T}(\varepsilon)+1} &= f_{\tilde{T}(\varepsilon)+1}(x_{\tilde{T}(\varepsilon)+1}, y_{\tilde{T}(\varepsilon)+1}) +\frac{d_{\tilde{T}(\varepsilon)+1}}{2}\|y_{\tilde{T}(\varepsilon)+1}-y_{\tilde{T}(\varepsilon)}\|^2- \frac{c_{\tilde{T}(\varepsilon)} - c_{\tilde{T}(\varepsilon)+1}}{2} \|y_{\tilde{T}(\varepsilon)+1}\|^2 \\
				&\geq \underline{F} - {\rho_{\tilde{T}(\varepsilon)}} \hat{\mathcal{D}}_y^2 \geq \underline{F} - 4{L} \hat{\mathcal{D}}_y^2 = \underline{H}. \label{eq:lower_bound}
			\end{split}
		\end{equation}
		Combining \eqref{eq:lower_bound} with \eqref{eq:sum_inequality}, we then obtain
		\begin{equation}
			\sum_{k=2}^{\tilde{T}(\varepsilon)} \frac{{\tau}}{{k}^{1/3}} \|\nabla \tilde{G}_{k+1}\|^2\leq\frac{1}{144\rho_k{k}^{1/3} }  \|\nabla \tilde{G}_{k+1}\|^2 = m_k c_{k}\|\nabla \tilde{G}_{k+1}\|^2\leq H_2 - \underline{H} +11 L \hat{\mathcal{D}}_y^2. \label{eq:final_inequality}
		\end{equation}
		According to the setting of $c_k$ and \eqref{eq:final_inequality}, when $k \geq 2$, we can get
		\begin{equation}
			\begin{split}
				\frac{\varepsilon^2}{4} &\leq \frac{H_2 - \underline{H} + 11L\hat{\mathcal{D}}_y^2}{\sum^{\tilde{T(\varepsilon)}}_{k=2} \frac{\tau}{k^{1/3}}} \\
				&\leq \frac{H_2 - \underline{H} + 11L \hat{\mathcal{D}}_y^2}{\tau(\tilde{T}(\varepsilon)^{2/3} - 1)},
			\end{split}
		\end{equation}
		and then
		\begin{equation}
			\tilde{T}(\varepsilon) \leq \left( \frac{H_2 - \underline{H} + 11L \hat{\mathcal{D}}_y^2}{\tau \varepsilon^2} + 1 \right)^{3/2}.
		\end{equation}
		On the other hand, when $k \geq \left( \frac{8L \hat{\mathcal{D}}_y}{\varepsilon} \right)^3 \geq \left( \frac{2\rho_k \hat{\mathcal{D}}_y}{\varepsilon} \right)^3$, the parameter setting of $c_k$ ensures
		\begin{equation}
			c_k \hat{\mathcal{D}}_y\leq \frac{\varepsilon}{2}.
		\end{equation}
		Therefore, there exists $T(\varepsilon) \leq \max \left( \tilde{T}(\varepsilon), \left( \frac{8L \hat{\mathcal{D}}_y}{\varepsilon} \right)^3 \right)$ such that
		\begin{equation}
			\| \nabla G_{k+1} \| \leq \| \nabla \tilde{G}_{k+1} \| + c_k \hat{\mathcal{D}}_y \leq \frac{\varepsilon}{2} + \frac{\varepsilon}{2} = \varepsilon,
		\end{equation}
		which completes the proof.
	\end{proof}
	\begin{remark}
		In the PF-AGP-NL algorithm, the upper bound of the total number of backtracking steps is \ {$\sum\limits_{k=2}^{T(\varepsilon)} \left(\log_2\left(\frac{l^k_{11}}{l_{11}^{k-1}}\right)+\log_2\left(\frac{l^{k}_{12}}{l_{12}^{k-1}}\right)\right)=\log_2\left(\frac{2L_{11}}{l_{11}^{1}}\right)+\log_2\left(\frac{2L_{12}}{l_{12}^{1}}\right)=\mathcal{O}(\log(L))$}, and the upper bound of the total number of line searches is $\mathcal{O}(L)$. According to Theorem \ref{thm:convergence_rate}, in the nonconvex-linear setting, the number of gradient evaluations required by the PF-AGP-NL algorithm to obtain an $\varepsilon$ -stationary point is $\mathcal{O}(L^3 \varepsilon^{-3})$.
	\end{remark}
	
	\section{Numerical results} \label{senu}
	In this section, we report some numerical results to show the efficiency of the proposed {PF-AGP-NSC} and {PF-AGP-NC} algorithms. All the numerical tests are implemented in Python 3.12 and run with an NVIDIA RTX 4090D GPU (24GB).
	
	\subsection{Synthetic Minimax Problem}
	Synthetic minimax problem can be formulated as the following nonconvex-strongly concave minimax problem \citep{luo2022finding}:
	\begin{equation}
		\min_{\mathbf{x} \in \mathbb{R}^3} \max_{\mathbf{y} \in \mathbb{R}^2} f(\mathbf{x}, \mathbf{y}) = w(x_3) - \frac{y_1^2}{40} + x_1 y_1 - \frac{5 y_2^2}{2} + x_2 y_2, \label{testprob1}
	\end{equation}
	where $\mathbf{x} = [x_1, x_2, x_3]^\top$, $\mathbf{y} = [y_1, y_2]^\top$ and 
	{\small\begin{equation*}
			w(x) =
			\begin{cases} 
				\sqrt{\epsilon} \left( x + (\lambda + 1)\sqrt{\epsilon} \right)^2 - \frac{1}{3} \left( x + (\lambda + 1)\sqrt{\epsilon} \right)^3 - \frac{1}{3} (3\lambda + 1) \epsilon^{3/2}, & x \leq -\lambda\sqrt{\epsilon}; \\
				\epsilon x + \frac{\epsilon^{3/2}}{3}, & -\lambda\sqrt{\epsilon} < x \leq -\sqrt{\epsilon}; \\
				-\sqrt{\epsilon} x^2 - \frac{x^3}{3}, & -\sqrt{\epsilon} < x \leq 0; \\
				-\sqrt{\epsilon} x^2 + \frac{x^3}{3}, & 0 < x \leq \sqrt{\epsilon};\\
				-\epsilon x + \frac{\epsilon^{3/2}}{3}, & \sqrt{\epsilon} < x \leq \lambda\sqrt{\epsilon}; \\
				\sqrt{\epsilon} \left( x - (\lambda + 1)\sqrt{\epsilon} \right)^2 + \frac{1}{3} \left( x - (\lambda + 1)\sqrt{\epsilon} \right)^3 - \frac{1}{3} (3\lambda + 1) \epsilon^{3/2}, & \lambda\sqrt{\epsilon} \leq x.
			\end{cases}
	\end{equation*}}
	It is easy to verify that \eqref{testprob1} has a strict saddle point at $(\mathbf{x}^*, \mathbf{y}^*) = ([0, 0, 0]^\top, [0, 0]^\top)$. 
	
	\textbf{Experimental setup}. We compare the proposed {PF-AGP-NSC} algorithm with the SGDA-B algorithm, the NeAda algorithm, the TiAda algorithm and the AGP algorithm for solving \eqref{testprob1} with $\epsilon = 0.01$ and $\lambda= 5$. The initial point is chosen as $(x_0, y_0)=([0, 0, 2]^\top, [0, 0]^\top)$ for all four tested algorithms. For the {PF-AGP-NSC} algorithm, we set $l_{11}^{1}=0.01$ , $l_{12}^{1}=0.01$, $l_{22}^{1}=0.01$ and $\mu_1=0.01$. For the SGDA-B algorithm, we set $\mu = 1$, $\gamma = 0.9$, $N = 1$, $p=0.5$ and $\tilde{L} = \frac{10}{9}$. For the NeAda algorithm, we set $\eta_x=0.5$ and $\eta_y=0.1$. For the TiAda algorithm, we set $\alpha = 0.6$, $\beta = 0.4$, $\eta_x =0.8 $, $\eta_y = 0.4$, $v^x_0 = v^y_0 = 1$. For the AGP algorithm, we set $\alpha_x = 0.14$, $\beta_y = 1.1$. The algorithms terminate when $\|\nabla f(x_k, y_k)\|\le 10^{-5}$ satisfied.
	
	\begin{figure}[h]
		\centering
		\includegraphics[width=220pt]{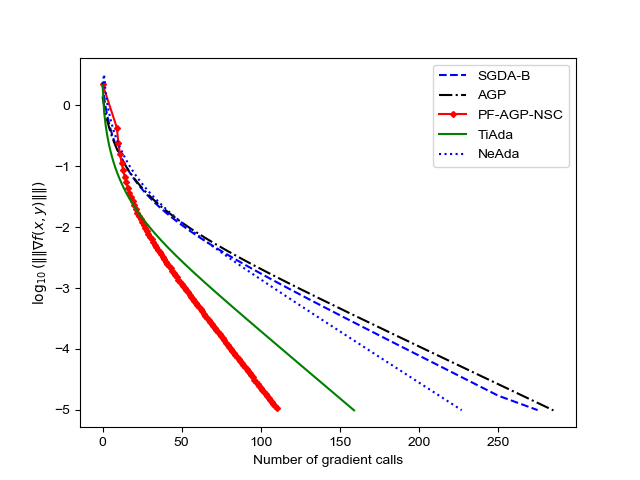}
		\caption{Numerical Performance of five tested algorithms for solving the Synthetic Minimax problem}
		\label{fig-NCSC}
	\end{figure}
	
	\textbf{Results}. Figure \ref{fig-NCSC} shows the stationarity gap of all four test algorithms, which demonstrates that the proposed {PF-AGP-NSC} algorithm exhibits significantly better convergence performance than that of the SGDA-B algorithm, TiAda algorithm and AGP algorithm.
	

	\subsection{Dirac-GAN problem}
	
	The Dirac-GAN problem can be formulated as the following nonconvex-concave minimax problem \citep{Mescheder2018}:
	\begin{equation}
		\min_x \max_y L(x, y) =-\log(1 + \exp(-xy)) + \log 2, 
	\end{equation}
	where $(0, 0)$ is the unique stationary point.
	
	\textbf{Experimental setup}.
	We compare the PF-AGP-NC algorithm with the SGDA-B algorithm and the AGP algorithm. The initial point of all algorithms is chosen as $(1, 1)$. For the {PF-AGP-NC} algorithm,  we set $l_{11}^{1}=0.01$ , $l_{12}^{1}=1$ and $l_{22}^{1}=0.01$.  For the SGDA-B algorithm, we set $\mu = 1$, $\gamma = 0.5$, $N = 1$, $p= 0.5$ and $\tilde{L}$ = 1. For the AGP algorithm, we set $\alpha_x = \frac{0.8}{\sqrt{k}}$, $\beta_y = 0.3$ and $c_k =\frac{0.5}{k^{1/4}}$.
	The algorithms terminate when $\|\nabla f(x_k, y_k)\|\le 10^{-5}$ satisfied.

	\textbf{Results}. Figure \ref{fig:example} shows the stationarity gap and the sequences generated by all three test algorithms. We find
	that the SGDA-B algorithm fails to converge, both the AGP algorithm and the {PF-AGP-NC} algorithms successfully approach the unique stationary point. Moreover, the {PF-AGP-NC} algorithm exhibits significantly better convergence performance than that of AGP algorithm.
	\begin{figure}[h]
		\centering
		\subfigure[]{
			\label{fig_c1}
			\includegraphics[width=160pt]{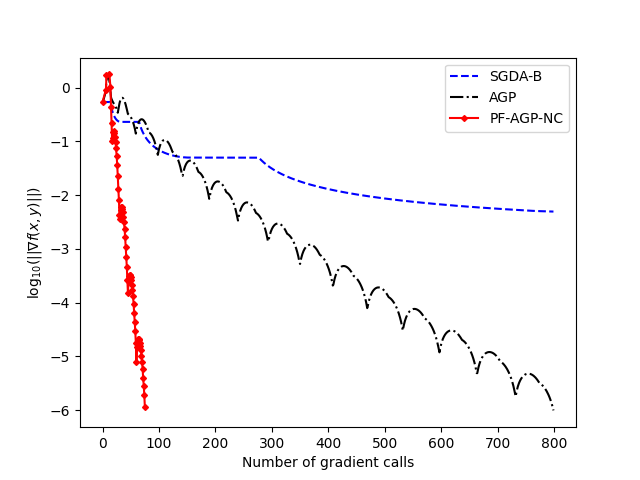}}
		\subfigure[]{
			\label{fig_d2}
			\includegraphics[width=160pt]{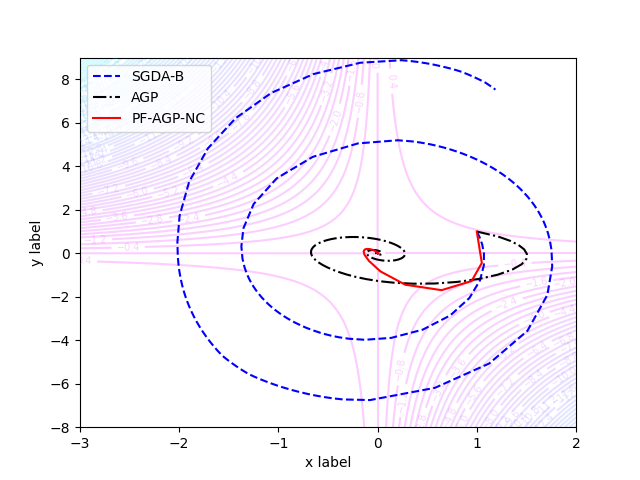}}
		\caption{Numerical performance of three tested algorithms for solving the Dirac-GAN problem}
		\label{fig:example}
	\end{figure}
	\subsection{\textbf{Robust learning over multiple domains}}
	In this subsection,  we consider the following robust learning problem over multiple domains \citep{qian2019robust}, formulated as a nonconvex-linear minimax problem:
	\begin{equation}
		\min _{x} \max _{y \in \Delta} y^{T} F(x),
	\end{equation}
	where $F(x):=\left[f_{1}(x) ; \ldots ; f_{M}(x)\right] \in \mathbb{R}^{M \times 1}$ with $f_{m}(x)=\frac{1}{\left|\mathcal{S}_{m}\right|} \sum_{i=1}^{\left|\mathcal{S}_{m}\right|} \ell\left(d_{i}^{m}, l_{i}^{m}, x\right)$ and $\ell$ can be any non-negative loss function,  $ M $ is the number of tasks and $ x $ denotes the network parameters, $ y \in \Delta$ describes the weights over different tasks and $ \Delta $ is the simplex, i.e., $\Delta=\left\{(y_1,\cdots,y_M)^{T} \mid \sum_{m=1}^{M} y_{m}=1,  y_{m} \geq 0\right\}$. 
	We consider two image classification problems with MNIST \citep{lecun1998gradient} and CIFAR10 datasets \citep{krizhevsky2009learning}. The dimension of  $x$ is 57,044,810. The sample sizes of the CIFAR-10 and MNIST datasets are both 60,000.
	Our goal is to train a neural network that works on these two completely unrelated problems simultaneously. Since cross-entropy loss is popular in multi-class classification problem, we use it as our loss function. We should point out that the quality of the algorithm for solving a robust learning problem over multiple domains is measured by the worst case accuracy over all domains \citep{qian2019robust}.

	
	\textbf{Experiment setup.} We compare the proposed {PF-AGP-NC} algorithm, PF-AGP-NL algorithm with the SGDA-B algorithm and the AGP algorithm. For the {PF-AGP-NC} algorithm,  we set $l_{11}^{1}=0.01$ , $l_{12}^{1}=0.01$ and $l_{22}^{1}=5$. For the PF-AGP-NL algorithm,  we set $l_{11}^{1}=0.1$ and $l_{12}^{1}=0.1$.  For the SGDA-B algorithm, we set $\mu = 20$, $\gamma = 0.9$, $N = 1$, $p=0.5$ and $\tilde{L}$ = 20. For the AGP algorithm, we set $ \frac{1}{\beta_{k}} =\frac{2}{2+\sqrt{k}} $, $\gamma_{k} =100$ and $ c_{k} =\frac{1}{10+k^{\frac{1}{4}}}$. We set the batch size in all tasks to be $ 128 $ and run $ 50 $ epochs for all algorithms. Moreover, we sample our results every epoch and calculate the average accuracy on these two testing sets to evaluate the performance of different algorithms. Due to hardware (especially memory) limitations, we use a mini-batch randomized gradient instead of the exact gradient at each iteration in our numerical experiment, which is the same as in \citep{xu2023unified}.
	\begin{figure}[h]
		\centering
		\subfigure[]{
			\label{fig_c}
			\includegraphics[width=160pt]{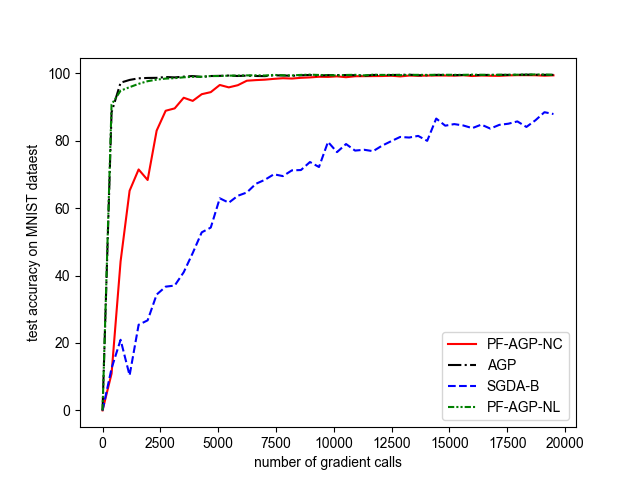}}
		\subfigure[]{
			\label{fig_d}
			\includegraphics[width=160pt]{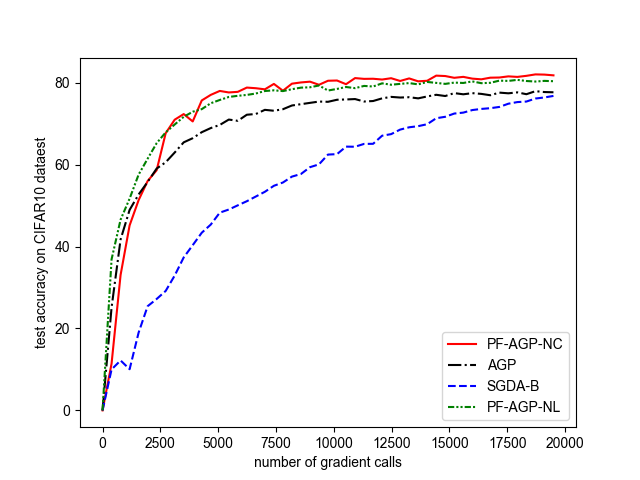}}
		\caption{Numerical Performance of four tested algorithms for solving robust multi-task learning problem.}
		\label{img1}
	\end{figure}
	
	\textbf{Results.}  Figure \ref{fig_c} and Figure \ref{fig_d} show the testing accuracy on the MNIST dataset and  the CIFAR10 dataset respectively. The PF-AGP-NL algorithm and PF-AGP-NC algorithm outperform the SGDA-B algorithm and the AGP algorithm for both datasets.

	\section{Conclusions}
	In this paper, based on the framework of backtracking, we propose two completely parameter-free alternating gradient projection algorithms, i.e., the {PF-AGP-NSC} algorithm and the {PF-AGP-NC} algorithm, for solving nonconvex-(strongly) concave minimax problems respectively, which does not require prior knowledge of parameters such as the Lipschtiz constant $L$ or the strongly concavity modulus  $\mu$.   Moreover, we show that the total number of gradient calls of the {PF-AGP-NSC} algorithm to obtain an $\varepsilon$-stationary point for nonconvex-strongly concave minimax problems is upper bounded by \ {$\mathcal{O}\left( L^2\kappa^3\varepsilon^{-2} \right)$} where $\kappa$ is the condition number,  while the total number of gradient calls of the {PF-AGP-NC} algorithm to obtain an $\varepsilon$-stationary point for nonconvex-concave minimax problems is upper bounded by \ {$\mathcal{O}\left( \log^2(L)L^4\varepsilon^{-4} \right)$}. 
	
	Moreover, we also propose a completely parameter-free alternating gradient projection algorithm, i.e., the {PF-AGP-NL} algorithm, for solving nonconvex-linear minimax problems, which does not require any prior knowledge of parameters.
	As far as we know, the PF-AGP-NSC algorithm, the PF-AGP-NC algorithm and  the PF-AGP-NL algorithm are the first completely parameter-free algorithms for solving nonconvex-strongly concave minimax problems, nonconvex-concave minimax problems and nonconvex-linear minimax problems respectively. Numerical results demonstrate the efficiency of the proposed algorithms.

	\bibliography{sample}
	\appendix
	
	\section{A restarted parameter-free single-loop  algorithm for nonconvex-strongly concave minimax problems}  \label{app1}
	In this section, we propose a restarting parameter-free alternating gradient projection algorithm (rPF-AGP-NSC) for the nonconvex-strongly concave minimax problem. The algorithm is a restarting variant of PF-AGP-NSC that does not require the prior knowledge of the Lipschtiz constant $L$, but requires the prior knowledge of two additional parameters $\mu$ and $\underline{S}$. The algorithm  is formally presented in Algorithm \ref{restart_al_nc_sc4}. 
	
	Next, we prove the iteration complexity of Algorithm \ref{restart_al_nc_sc4} under the nonconvex-strongly concave setting. Denote   
	$\Phi(x):=\max_{y\in\mathcal{Y}}f(x,y)$ and 
	$y^*(x):=\arg\max_{y\in\mathcal{Y}}f(x,y)$.
	By Lemma B.1 in \citep{Nouiehed}, $\Phi(x)$ is $L_\Phi$-Lipschitz smooth with $L_\Phi=L+\frac{L^2}{\mu}$. Moreover,  we have
	\begin{align}
		\nabla_x\Phi(x)=\nabla_x f (x, y^*(x)) \label{gradla:nsc}.
	\end{align}
	\begin{algorithm}[t]
		\caption{A restarted parameter-free alternating gradient projection (rPF-AGP-NSC) algorithm for nonconvex-strongly concave minimax problems}
		\begin{algorithmic}
			\STATE \textbf{Step 1:} Input  $x_{1,1}, y_{1,1}$, $l_{1}$,  $\mu$, $\varepsilon$, $\underline{S}$; Set $i=1$. Compute: $S_{1}=2\max_{y \in \mathcal{Y}}f(x_{1,1},y)-f(x_{1,1}, y_{1,1})$. 
			\STATE \textbf{Step 2:} Comppute $\beta_i= 3l_i$, $\alpha_i=\frac{158l_i^3}{\mu^2}$, $\eta_i  =\frac{(2\beta_i+\mu)(\beta_i+l_i)}{\mu\beta_i}$, 
			$\tilde{d}^i_1 =\min \{\frac{\alpha_i-\frac{l_i(l_i+\beta_i)^2\eta_i^2}{\beta_i^2}					-\frac{l_i^2}{\mu}-\frac{3l_i}{2}}{2\alpha_i^2},\frac{ \beta_i-\frac{3l_i}{2}}{\beta_i^2+2l_i^{2}} \}$, $ \tilde{T}_i = \frac{S_{1}-\underline{S}}{\varepsilon^2 \tilde{d}^i_1}$;  Set $k =1$; 
			\STATE \textbf{Step 3:} {\bf Update $x_{i,k+1}$ and $y_{i,k+1}$}:
			\STATE
			\quad  \textbf{(a):}   \textbf{If} $k \geq \tilde{T}_i$ \textbf{then}, go to \textbf{Step 4};
			\STATE	\quad  \textbf{(b):}  Update $x_{i,{k+1}}$ and $y_{i,{k+1}}$
			\begin{align}
				y_{i,{k+1}} &=  \mathcal{P}_\mathcal{Y}\left(y_{i,k} + \frac{1}{\beta_i} \nabla_y f(x_{i,{k}}, y_{i,k})\right), \label{update-ys}\\
				x_{i,{k+1}} &=  \mathcal{P}_\mathcal{X}\left(x_{i,k} - \frac{1}{\alpha_i} \nabla_x f(x_{i,k}, y_{i,k+1}) \right)\label{update-xs}.
			\end{align}
			\STATE	\quad  \textbf{(c):}   Set $k = k+1$ go to \textbf{Step 3}.
			\STATE \textbf{Step 4:} \textbf{If} $\|\nabla G(x_{i,k}, y_{i,k})\| \leq \varepsilon$, \textbf{stop}; \textbf{Otherwise} $x_{i+1,1}=x_{1,1}$, $y_{i+1,1}=y_{1,1}$, $l_{i+1} = 2l_i$, $i=i+1$, go to \textbf{Step 2}. \\ 	
		\end{algorithmic}
		\label{restart_al_nc_sc4}
	\end{algorithm}
	\begin{lemma}\label{lem1}
		Suppose that Assumption \ref{as} holds. Let $\{\left(x_{i,k},y_{i,k}\right)\}$ be a sequence generated by Algorithm \ref{restart_al_nc_sc4}. Let $\eta_i=\frac{(2\beta_i+\mu)(\beta_i+L)}{\mu\beta_i}$, then $\forall k,i \ge 1$, and for any $a>0$,
		\begin{align}\label{lem1:iq1}
			&\Phi(x_{i,k+1}) -\Phi(x_{i,k})\nonumber \\
			\leq& \langle  \nabla _x f( x_{i,k},y_{i,k+1}),x_{i,k+1}-x_{i,k} \rangle+\frac{L^2(L+\beta_i)^2\eta_i^2}{2a\beta_i^2}\|y_{i,k+1}-y_{i,k}\|^2 +\frac{a+L_\Phi}{2}\|x_{i,k+1}-x_{i,k}\|^2.
		\end{align}
	\end{lemma}
	\begin{proof}
		Since $\Phi(x)$ is $L_\Phi$ smooth with respect to $x$, and by \eqref{gradla:nsc}, we have
		\begin{align} \label{lem1:2}
			\Phi(x_{i,k+1})-\Phi(x_{i,k})
			\le&\langle \nabla _xf(x_{i,k},y^*(x_{i,k}))-\nabla _xf(x_{i,k},y_{i,k+1}) ,x_{i,k+1}-x_{i,k} \rangle  \nonumber\\
			&+\frac{L_\Phi}{2}\|x_{i,k+1}-x_{i,k}\|^2+\langle \nabla _xf(x_{i,k},y_{i,k+1}) ,x_{i,k+1}-x_{i,k} \rangle.
		\end{align}
		By the Cauchy-Schwarz inequality, for any $a>0$, we obtain
		\begin{align}
			&\langle \nabla _xf(x_{i,k},y^*(x_{i,k}))-\nabla _xf(x_{i,k},y_{i,k+1}) ,x_{i,k+1}-x_{i,k} \rangle\nonumber\\
			\le&\frac{L^2}{2a}\|y_{i,k+1}-y^*(x_{i,k})\|^2+\frac{a}{2}\|x_{i,k+1}-x_{i,k}\|^2.\label{lem1:3}
		\end{align}
		By Lemma 1 in \citep{zhang2022primal}, we have
		\begin{align}\label{lem1:4}
			&\|y_{i,k+1}-y^*(x_{i,k})\|
			\le \eta_i \left\| y_{i,k+1}-\mathcal{P}_{\mathcal{Y}}\left( y_{i,k+1}+ \frac{1}{\beta_i} \nabla _yf\left( x_{i,k},y_{i,k+1} \right) \right)\right\|.
		\end{align}
		By \eqref{update-ys}, the non-expansive property of the projection operator $\mathcal{P}_{\mathcal{Y}}(\cdot)$ and Assumption \ref{as}, we obtain
		\begin{align}\label{lem1:5}
			&\left\| y_{i,k+1}-\mathcal{P}_{\mathcal{Y}}\left( y_{i,k+1}+ \frac{1}{\beta_i} \nabla _yf\left( x_{i,k},y_{i,k+1} \right) \right)\right\|
			\le\frac{L+\beta_i}{\beta_i}\|y_{i,k+1}-y_{i,k}\|.
		\end{align}
		Combing \eqref{lem1:2}, \eqref{lem1:3}, \eqref{lem1:4} and \eqref{lem1:5}, we get
		\begin{align}\label{lem1:6}
			\Phi(x_{i,k+1})-\Phi(x_{i,k})
			\le&\frac{L^2\eta_i^2(L+\beta_i)^2}{2a\beta_i^2}\|y_{i,k+1}-y_{i,k}\|^2+\frac{a+L_\Phi}{2}\|x_{i,k+1}-x_{i,k}\|^2\nonumber\\
			&+\langle \nabla _x f(x_{i,k},y_{i,k+1}) ,x_{i,k+1}-x_{i,k} \rangle,
		\end{align}
		which completes the proof.
	\end{proof}
	Next, we provide an lower bound for the difference between $f(x_{i,k+1},y_{i,k+1})$ and $f(x_{i,k},y_{i,k})$.
	\begin{lem}\label{lem2}
		Suppose that Assumption \ref{as} holds. Let $\{\left(x_{i,k},y_{i,k}\right)\}$ be a sequence generated by Algorithm \ref{restart_al_nc_sc4}.
		Then $\forall i,k \ge 1$,
		\begin{align}\label{lem2:iq1}
			&f(x_{i,k+1},y_{i,k+1})-f( x_{i,k},y_{i,k})\nonumber \\
			\geq&\left(\beta_i-\frac{L}{2}\right)\|y_{i,k+1}-y_{i,k}\|^2-\frac{L}{2}\| x_{i,k+1}-x_{i,k} \|^2
			+\langle \nabla _{x}f(x_{i,k},y_{i,k+1}),x_{i,k+1}-x_{i,k} \rangle,
		\end{align}.
	\end{lem}
	\begin{proof}
		The optimality condition for $y_{i,k}$ in \eqref{update-ys} implies that $\forall y\in \mathcal{Y}$ and $\forall i,k\geq 1$,
		\begin{equation}\label{lem2:3}
			\langle \nabla _y f(x_{i,k},y_{i,k}),y_{i,k+1}-y_{i,k} \rangle \ge \beta_i\|y_{i,k+1}-y_{i,k}\|^2.
		\end{equation}
		By  Assumption \ref{as}, the gradient of $f(x,y)$ is Lipschitz continuous with respect to $y$ and  \eqref{lem2:3}, we get
		\begin{align}\label{lem2:4}
			f(x_{i,k},y_{i,k+1})-f( x_{i,k},y_{i,k})
			\ge&\left(\beta_i-\frac{L}{2}\right)\|y_{i,k+1}-y_{i,k}\|^2.
		\end{align}
		Since that the gradient of $f(x,y)$ is Lipschitz continuous with respect to $x$, we obtain
		\begin{align}\label{lem2:5}
			& f(x_{i,k+1},y_{i,k+1})-f( x_{i,k},y_{i,k+1})\nonumber\\
			\ge & \langle \nabla _{x}f(x_{i,k},y_{i,k+1}),x_{i,k+1}-x_{i,k} \rangle -\frac{L}{2}\| x_{i,k+1}-x_{i,k} \|^2.
		\end{align}
		The proof is then completed by adding \eqref{lem2:4} and \eqref{lem2:5}.
	\end{proof}
	We now establish an important recursion for the Algorithm \ref{restart_al_nc_sc4}.
	\begin{lem}\label{lem3}
		Suppose that Assumption \ref{as} holds. Denote
		$
		S(x,y)=2\Phi(x)-f(x,y).
		$
		Let $\{\left(x_{i,k},y_{i,k}\right)\}$ be a sequence generated by Algorithm \ref{restart_al_nc_sc4}. Then $\forall i,k \geq 1$,
		\begin{align}\label{lem3:1}
			S(x_{i,k+1},y_{i,k+1})-S( x_{i,k},y_{i,k})
			\le&-\left(\alpha_i-\frac{L(L+\beta_i)^2\eta_i^2}{\beta_i^2}
			-\frac{L^2}{\mu}-\frac{3L}{2}\right)\|x_{i,k+1}-x_{i,k}\|^2 \nonumber \\
			&-\left(\beta_i-\frac{3L}{2}\right)\|y_{i,k+1}-y_{i,k}\|^2.
		\end{align}
	\end{lem}
	\begin{proof}
		Denote 	$S(x,y)=2\Phi(x)-f(x,y)$. 
		Combining \eqref{lem1:iq1} and  \eqref{lem2:iq1}, we have
		\begin{align}\label{lem3:2}
			&S(x_{i,k+1},y_{i,k+1})-S( x_{i,k},y_{i,k})\nonumber \\
			\le&-\left(\beta_i-\frac{L}{2}-\frac{(L+\beta_i)^2\eta_i^2L^2}{a\beta_i^2}\right)\|y_{i,k+1}-y_{i,k}\|^2+\left(a+L_\Phi+\frac{L}{2}\right)\|x_{i,k+1}-x_{i,k}\|^2 \nonumber \\
			&+\langle  \nabla _x f( x_{i,k},y_{i,k+1}),x_{i,k+1}-x_{i,k}\rangle.
		\end{align}
		The optimality condition for $x_{i,k}$ in \eqref{update-xs} implies that $\forall x\in \mathcal{X}$ and $\forall k\geq 1$,
		\begin{equation}\label{lem3:4}
			\langle \nabla _xf(x_{i,k},y_{i,k+1}),x_{i,k+1}-x_{i,k} \rangle \le -\alpha_i\|x_{i,k+1}-x_{i,k}\|^2.
		\end{equation}
		Plugging \eqref{lem3:4} into \eqref{lem3:2}, then we obtain
		\begin{align}\label{lem3:6}
			&S(x_{i,k+1},y_{i,k+1})-S( x_{i,k},y_{i,k})\nonumber \\
			\le&-\left(\beta_i-\frac{L}{2}-\frac{L^2(L+\beta_i)^2\eta^2}{a\beta_i^2}\right)\|y_{i,k+1}-y_{i,k}\|^2 -\left(\alpha_i-a-L_\Phi-\frac{L}{2}\right)\|x_{i,k+1}-x_{i,k}\|^2.
		\end{align}
		The proof is then completed by the definiton of $L_\Phi=L+\frac{L^2}{\mu}$ and choosing $a=\frac{L(L+\beta_i)^2\eta_i^2}{\beta_i^2}$ in \eqref{lem3:6}.
	\end{proof}
	We are now ready to establish the iteration complexity for the {rPF-AGP-NSC}  algorithm. In particular, letting $\nabla G(x_{i,k}, y_{i,k})$ be defined as in Definition \ref{gap} and $\varepsilon > 0$ be a given target accuracy, we provide a bound on $T_i(\varepsilon)$, the  inner iteration index to achieve an $\varepsilon$-stationary point, i.e., $\|\nabla G(x_{i,k}, y_{i,k})\| \leq \varepsilon$, which is equivalent to
	\begin{equation}\label{r_alg_sp}
		T_i(\varepsilon) = \min\{k \mid \|\nabla G(x_{i,k}, y_{i,k})\| \leq \varepsilon\}.
	\end{equation}
	The following theorem gives an upper bound on $T_i(\varepsilon)$.  
	\begin{thm}\label{thm1}
		Suppose that Assumptions  \ref{as} hold. Let $\{\left(x_{i,k},y_{i,k}\right)\}$ be a sequence generated by Algorithm \ref{restart_al_nc_sc4}. Let $\eta_i=\frac{(2\beta_i+\mu)(\beta_i+L)}{\mu\beta_i}$. If
		\begin{align}\label{thm1:1}
			&\beta_i>\frac{3L}{2}, \alpha_i>\frac{L(L+\beta_i)^2\eta_i^2}{\beta_i^2}
			+\frac{L^2}{\mu}+\frac{3L}{2},
		\end{align}
		then $\forall \varepsilon>0$, we have
		$$ T_i\left( \varepsilon \right) \le \frac{d_2^i}{\varepsilon ^2 d^i_1},$$
		where $d^i_1:=\min\{{\frac{\alpha_i-\frac{L(L+\beta_i)^2\eta_i^2}{\beta_i^2}
				-\frac{L^2}{\mu}-\frac{3L}{2}}{2\alpha_i^2},\frac{ \beta_i-\frac{3L}{2}}{\beta_i^2+2L^{2}} \} }$, $d_2^i:=S(x_{i,1},y_{i,1})-\underbar{S}$ with $\underline{S} := \min\limits_{x\in\mathcal{X}}\max\limits_{y\in \mathcal{Y}}f(x,y)$.
	\end{thm}
	\begin{proof}
		By \eqref{update-ys}, we immediately obtain
		\begin{equation}\label{lem4:2}
			\left\| \beta_i\left( y_{i,k}-\mathcal{P}_{\mathcal{Y}}\left( y_{i,k}+ \frac{1}{\beta_i} \nabla _yf\left( x_{i,k}, y_{i,k},\right) \right) \right) \right\|=\beta_i\|y_{i,k+1}-y_{i,k}\|.
		\end{equation}
		On the other hand, by \eqref{update-xs} and the Cauchy-Schwartz inequality, we have
		\begin{align}\label{lem4:3}
			&\|	\alpha_i ( x_{i,k}-\mathcal{P}_{\mathcal{Y}}( x_{i,k}-\frac{1}{\alpha_i}\nabla _xf\left( x_{i,k}, y_{i,k} \right) ) )\|
			\leq  \alpha_i\|x_{i,k+1}-x_{i,k}\| +L\|y_{i,k+1}-y_{i,k}\|,
		\end{align}
		where the last inequality is by the nonexpansive property of the projection operator and Assumption \ref{as}. 
		Combing \eqref{lem4:2}, \eqref{lem4:3}, and using Cauchy-Schwarz inequality, we have
		\begin{align}\label{lem4:1}
			\|\nabla G( x_{i,k}, y_{i,k})\|^2
			\le (\beta_i^2+2L^{2})\|y_{i,k+1}-y_{i,k}\|^2+2\alpha_i^2\|x_{i,k+1}-x_{i,k}\|^2.
		\end{align}
		By \eqref{thm1:1}, it can be easily checked that $d^i_1>0$. By multiplying $d^i_1$ on the both sides of \eqref{lem4:1} and using \eqref{lem3:1} in Lemma \ref{lem3}, we get
		\begin{align}
			d^i_1\|\nabla G( x_{i,k}, y_{i,k})\|^2
			\le S(x_{i,k},y_{i,k})-S(x_{i,k+1},y_{i,k+1}).\label{thm1:5}
		\end{align}
		Summing both sides of \eqref{thm1:5} from $k=1$ to $T_i(\varepsilon)$, we then obtain
		\begin{align}
			&\sum_{k=1}^{T_i\left( \varepsilon \right)}{d^i_1\|\nabla G( x_{i,k}, y_{i,k}) \| ^2}
			\le S(x_{i,1},y_{i,1})-S(x_{i,T_i(\varepsilon)+1},y_{i,T_i(\varepsilon)+1}).\label{thm1:6}
		\end{align}
		Then, by \eqref{thm1:6} we obtain
		$
		\sum_{k=1}^{T_i\left( \varepsilon \right)}{d^i_1\|\nabla G( x_{i,k}, y_{i,k}) \| ^2}\le S(x_{i,1},y_{i,1})-\underbar{S}=d_2^i.
		$
		In view of the definition of $T_i(\varepsilon)$, the above inequality implies that $\varepsilon ^2\le \frac{d_2^i}{T_i( \varepsilon )d^i_1}$ or equivalently, $T_i\left( \varepsilon \right) \le \frac{d_2^i}{\varepsilon ^2 d^i_1}$. This completes  the proof. 
	\end{proof}
	\begin{remark}
		Let $ \beta_i = 3l_i, \alpha_i = \frac{158l_i^3}{\mu^2}$, then the outer iteration of the rPF-AGP-NSC algorithm requires at most $\mathcal{O}(\log(L))$ steps, so that $ \beta_i$ and $\alpha_i$ satisfy \eqref{thm1:1}. According to Theorem \ref{thm1}, the inner iteration of $\|\nabla G(x_{i,k}, y_{i,k})\| \leq \varepsilon$ requires at most $\frac{d_2^i}{\varepsilon ^2 d^i_1}$ steps.
	\end{remark}
	\begin{remark}
		Denote $\kappa=L/\mu$. If we set $\beta_i=3l_i$, by Theorem \ref{thm1} we have that $\eta_i=\mathcal{O}(\kappa)$, $d_1^i=\mathcal{O}(\frac{1}{L\kappa^2})$ and $d_2^i = O(1)$, which implies the total number of gradient calls  is upper bounded by $\sum\limits_{i=1}^{\log_2(L)+1} T_i (\varepsilon)=\mathcal{O}(\log(L)L\kappa^2\varepsilon^{-2})$ under the nonconvex-strongly concave setting. It means that if two additional parameters $\mu$ and $\underline{S}$ are known in advance, the gradient complexity of the rPF-AGP-NSC algorithm is improved to $\mathcal{O}\left(\log(L) L\kappa^2\varepsilon^{-2}\right)$, which is better than that of the SGDA-B and the PF-AGP-NSC algorithm. 
	\end{remark}

	\newpage
	
	\appendix
	
	
	
	
	
	
	

\end{document}